%% file: J_and_L.tex
\documentclass[journal]{IEEEtran}

\usepackage{amsmath} 
\usepackage{amssymb}  
\usepackage{amsthm}
\usepackage{color}
\usepackage{graphicx}
\usepackage{enumitem}
\usepackage{hyperref}
\usepackage{cite}
\makeatletter
\newenvironment{subtheorem}[1]{%
  \def\subtheoremcounter{#1}%
  \refstepcounter{#1}%
  \protected@edef\theparentnumber{\csname the#1\endcsname}%
  \setcounter{parentnumber}{\value{#1}}%
  \setcounter{#1}{0}%
  \expandafter\def\csname the#1\endcsname{\theparentnumber\alph{#1}}%
  \expandafter\def\csname theH#1\endcsname{thm.\theparentnumber\alph{#1}}%
  \unskip\ignorespaces
}{%
  \setcounter{\subtheoremcounter}{\value{parentnumber}}%
  \ignorespacesafterend
}
\makeatother
\newcounter{parentnumber}

\newtheorem{theorem}{Theorem}
\newtheorem{lemma}{Lemma}
\newtheorem{definition}{Definition}
\newtheorem{corollary}{Corollary}
\newtheorem{remark}{Remark}

\title{\LARGE \bf
Phase limitations of multipliers at harmonics
}

\author{William P. Heath,~\IEEEmembership{Member,~IEEE,} Joaquin Carrasco,~\IEEEmembership{Member,~IEEE,} and Jingfan Zhang,~\IEEEmembership{Student Member,~IEEE}
\thanks{WIlliam P. Heath, Joaquin Carrasco and Jingfan Zhang are all with the Control Systems Centre, Department of Electrical and Electronic Engineering, University of Manchester, UK
        {\tt\small william.heath@manchester.ac.uk; joaquin.carrascogomez@manchester.ac.uk, jingfan.zhang@manchester.ac.uk}}%
}

\begin{document}

\maketitle
\thispagestyle{empty}
\pagestyle{empty}

\begin{abstract}
We present a phase condition under which there is no suitable multiplier for a given continuous-time plant. The condition can be derived from either the duality approach or from the frequency interval approach. The condition has a simple graphical interpretation, can be tested in a numerically efficient manner and may be applied systematically. Numerical examples show significant improvement over existing results in the literature. The condition is used to demonstrate a third order system with delay that is a counterexample to the Kalman Conjecture.
\end{abstract}







\input{Part1_Introduction01}

\input{Part2_Preliminaries01}
\input{Part3_PhaseLimitations01}

\input{Part3a_Intervals}

\input{Part4_Examples01}

\section{Conclusion}

We have presented a simple graphical test that can rule out the existence of suitable OZF multipliers. The test can be implemented efficiently and systematically. The graphical interpretations provide considerable insight to the frequency behaviour of the OZF multipliers. Results show significantly improved results over those in the literature. The test  can be derived either from  the duality approach \cite{Jonsson:96,Jonsson96thesis,Jonsson97,Jonsson99} or from the frequency interval approach \cite{Megretski:95,Wang:18}. 

Guaranteeing there is no suitable OZF multiplier does not necessarily imply a Lurye system is not absolutely stable, although we have conjectured this to be the case \cite{Carrasco:EJC,Wang:18}. Kong and Su \cite{Khong20} show that  the implication is true with a wider class of nonlinearity; for this case the results of this paper may be applied directly. 
For the discrete-time case, Seiler and Carrasco \cite{Seiler21} provide a construction, for certain phase limitations, of a nonlinearity within the class for which the discrete-time Lurye system has a periodic solution. However the conjecture remains open for both continuous-time and discrete-time systems.

More generally results for discrete-time systems are quite different. For discrete-time systems an FIR search for multipliers is effective and outperforms others \cite{Wang:TAC}. With the interval approach  it is possible to find a nontrivial threshold such that the phase of a multiplier cannot be above the threshold over a certain frequency inteval \cite{Wang:18}. The duality approach leads to both a simple graphical test at simple frequencies and a condition at multiple frequencies that can be tested by linear program \cite{Zhang:20}.

This paper's results are for continuous-time single-input single-output multipliers of \cite{Zames68}.  Although multivariable extensions of the OZF multipliers are considered in the literature \cite{Safonov2000, DAmato2001, Kulkarni2002,Mancera2005, Fetzer2017}, it remains open what restrictions there might be. Similarly more general nonlinearities can be addressed with a reduced subset of the OZF multipliers \cite{Rantzer2001, Materassi11, Altshuller:13, Heath2021} and the analysis of this paper might be generalised to such cases. 
It also remains open whether a systematic procedure can be found with more points or intervals.

\input{app_proofs}

\bibliographystyle{IEEEtran}

\bibliography{harmonics_bib_trans}

\begin{IEEEbiography}[{\includegraphics[width=1in,height=1.25in,clip,keepaspectratio]{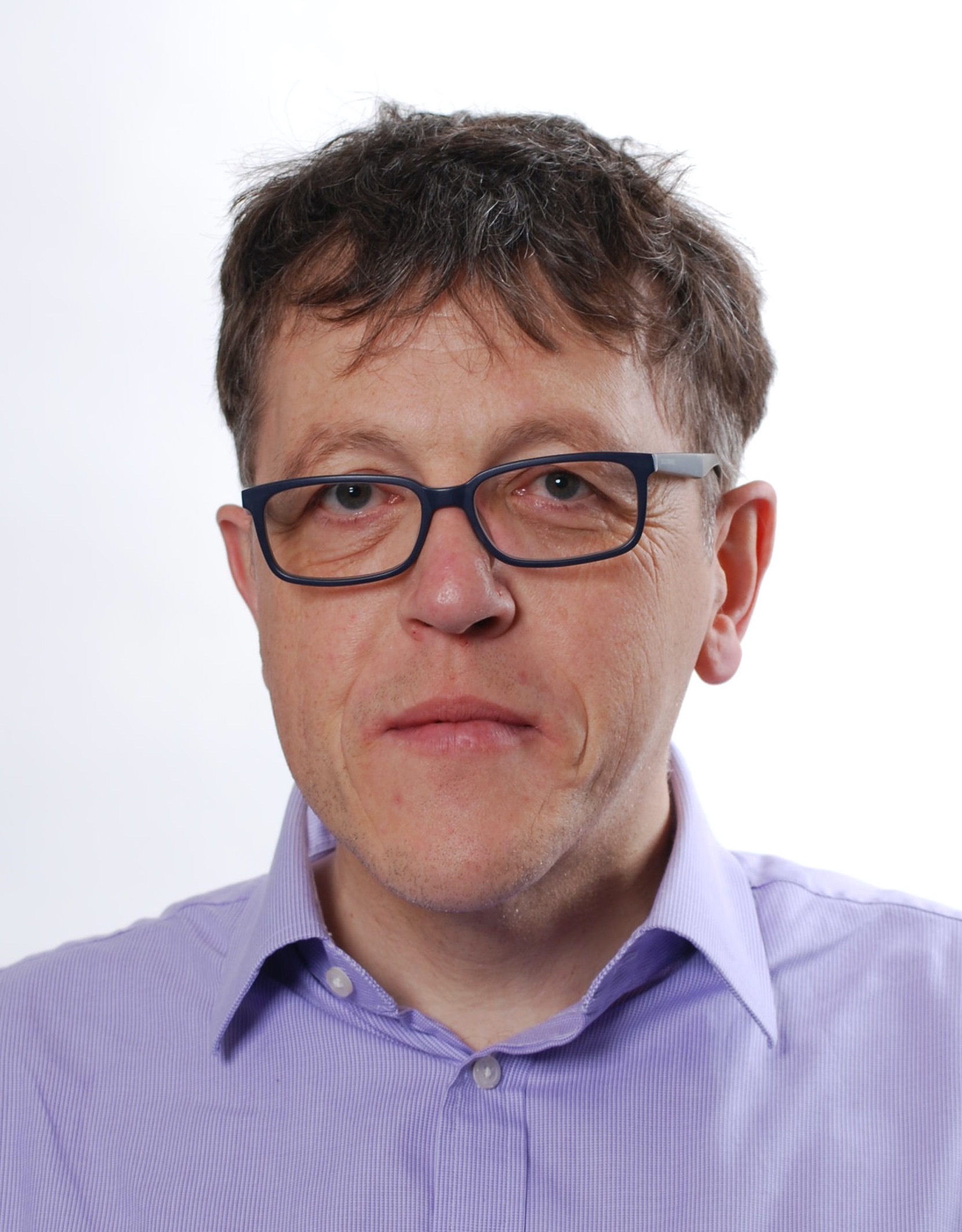}}]{William P. Heath}
 received an M.A. in mathematics from the University of Cambridge, U.K. and both an M.Sc. and Ph.D. in systems and control from the University of Manchester Institute of Science and Technology, U.K.  He is Chair of Feedback and Control with the Control Systems Centre and Head of the Department of Electrical and Electronic Engineering, University of Manchester, U.K.  Prior to joining the University of Manchester, he worked at Lucas Automotive and was a Research Academic at the University of Newcastle, Australia. His research interests include absolute stability, multiplier theory, constrained control, and system identification. 
\end{IEEEbiography}

\begin{IEEEbiography}[{\includegraphics[width=1in,height=1.25in,clip,keepaspectratio]{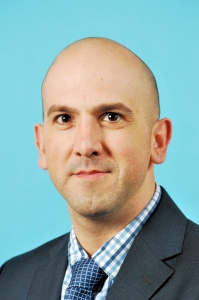}}]{Joaquin Carrasco}
                is a Reader at the Control Systems Centre, Department of Electrical and Electronic Engineering, University of Manchester, UK. He was born in Abarán, Spain, in 1978. He received the B.Sc. degree in physics and the Ph.D. degree in control engineering from the University of Murcia, Murcia, Spain, in 2004 and 2009, respectively. From 2009 to 2010, he was with the Institute of Measurement and Automatic Control, Leibniz Universität Hannover, Hannover, Germany. From 2010 to 2011, he was a research associate at the Control Systems Centre, School of Electrical and Electronic Engineering, University of Manchester, UK. His current research interests include absolute stability, multiplier theory, and robotics applications.
\end{IEEEbiography}

\begin{IEEEbiography}[{\includegraphics[width=1in,height=1.25in,clip,keepaspectratio]{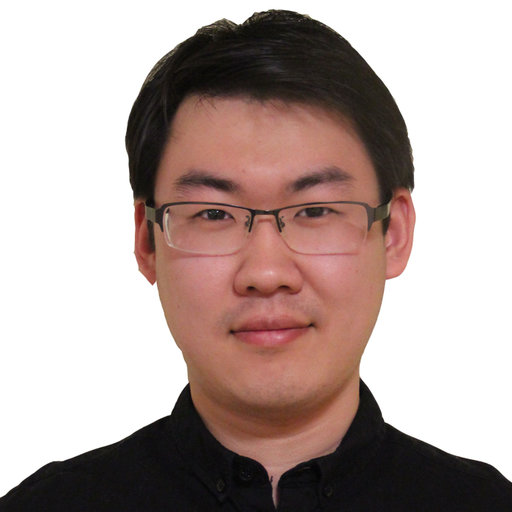}}]{Jingfan Zhang}
 received the B.Eng. degree in electrical engineering and its automation from Xi’an Jiaotong-Liverpool University, Suzhou, China, in 2015, and the M.Sc. degree in advanced control and systems engineering from the University of Manchester, Manchester, U.K., in 2016. He is currently working toward the Ph.D. degree with the Department of Electrical and Electronic Engineering, University of Manchester, Manchester, U.K. His research interests include absolute stability and applications of control theory in robotics.
\end{IEEEbiography}

\end{document}

%% file: Part1_Introduction01.tex
\section{Introduction}

The continuous-time OZF (O'Shea-Zames-Falb) multipliers were  discovered by O'Shea~\cite{OShea67} and formalised by Zames and Falb \cite{Zames68}. They preserve the positivity of  monotone memoryless nonlinearities. Hence they can be used, via loop transformation, to establish the absolute stability of Lurye systems with slope-restricted memoryless nonlinearities. An overview is given in \cite{Carrasco:EJC}.

Recent interest is largely driven by their compatability with the integral quadratic constraint (IQC) framework of Megretski and Rantzer \cite{Megretski97} and the availability of computational searches \cite{Safonov:87,  Gapski:94,Chang:12,Chen:95,Chen:96,Turner2009,Carrasco12,Turner:12,Carrasco:14}. A modification of the search proposed in \cite{Chen:95} is used in the Matlab IQC toolbox \cite{Kao:04} and analysed by Veenman and Scherer \cite{Veenman14}.

No single search method outperforms the others, and often a hand-tailored search outperforms an automated search \cite{Carrasco:14}. This motivates the analysis of conditions where a multiplier cannot exist. There are two main approaches in the literature. 
J\"{o}nsson and Laiou \cite{Jonsson:96} give a condition that must be satisfied at a number of isolated frequencies.  Their result is a particular case of a more general analsysis based on duality in an optimization framework \cite{Jonsson96thesis,Jonsson97,Jonsson99};  we will refer to this as the ``duality approach.'' Their result requires a non-trivial search over a finite number of parameters.
By contrast Megretski \cite{Megretski:95} gives a threshold such that the phase of a multiplier cannot be simultaneously above the threshold over a certain frequency interval and below its negative value on another. The idea is generalised in \cite{Wang:18}, where in particular the threshold for the second interval is allowed to have a different value. We will refer to this as the ``frequency interval approach.''
Both  the duality approach and the frequency interval  approach lead to powerful and useful results, but neither allows a systematic approach. With respect to the duality approach 
J\"{o}nsson states \cite{Jonsson96thesis} “it is in most applications hard to find a suitable frequency grid for the application of the results.''  With respect to the interval approach, in  \cite{Wang:18} we conclude that the most insightful choice of interval remains open. 

In this paper we present a simple phase condition on two frequencies whose ratio is rational. The condition can be be tested systematically. At each frequency ratio the condition leads to a graphical criterion similar to the off-axis circle criterion \cite{Cho:68} in that it can be expressed as a bound on the phase of a transfer function. We derive the condition via the duality approach, but we also show that it is equivalent to a limiting case of the frequency interval approach.
We illustrate the criterion on three examples: we show it gives a significantly better results for the numerical example in \cite{Jonsson:96}; we show it gives new bounds for the gain with O'Shea's classical example \cite{OShea67,Carrasco:EJC}; we provide an example of a third order transfer function with delay that does not satisfy the Kalman Conjecture.


The structure of this paper as follows. Section~\ref{Prem}  provides the necessary background material and includes the following minor contribution:  Theorems`\ref{Jthm_a} and~\ref{Jthm_b} provide frequency conditions  similar in spirit to the duality approach of \cite{Jonsson:96}, but more widely applicable; specifically the  conditions allow both the system transfer function and the multiplier to be irrational. 
The main results of the paper are presented in Section~\ref{Main}. Theorems~\ref{thm:2a} and~\ref{thm:2b} give a phase condition that has a simple graphical interpretation and can be implemented systematically. 
We prove Theorems~\ref{thm:2a} and~\ref{thm:2b} via the duality approach. We discuss both the graphical interpretation and the numerical implementation of Theorems~\ref{thm:2a} and~\ref{thm:2b}.  In Section~\ref{Int} we show that the results can also be derived via the frequency interval approach: Corollaries~\ref{m_corollary_a} and~\ref{m_corollary_b} provide a version of the interval approach \cite{Wang:18} for the limiting case where the length of interval goes to zero; Theorems~\ref{Meg_equiv_a} and~\ref{Meg_equiv_b}  state these corollaries are respectively equivalent to Theorems~\ref{thm:2a} and~\ref{thm:2b}.
Section~\ref{Exa} includes three examples: the first shows we achieve improved results over those reported in \cite{Jonsson:96}; the second is the benchmark problem of O'Shea\cite{OShea67} where we obtain improved results over those reported in \cite{Wang:18}; finally, in the third, we show that  a third order with delay system provides a counterexample to the Kalman Conjecture. All proofs, where not immediate, are given in the Appendix.

%% file: Part2_Preliminaries01.tex
\section{Preliminaries}\label{Prem}

\subsection{Multiplier theory}

We are concerned with the input-output stability of the Lurye system given by
\begin{equation}
y_1=Gu_1,\mbox{ } y_2=\phi u_2,\mbox{ } u_1=r_1-y_2 \mbox{ and }u_2 = y_1+r_2.\label{eq:Lurye}
\end{equation}
 Let $\mathcal{L}_2$ be the space of finite energy Lebesgue integrable signals and let $\mathcal{L}_{2e}$ be the corresponding extended space (see for example  \cite{desoer75}). The Lurye system is said to be stable if $r_1,r_2\in\mathcal{L}_2$  implies $u_1,u_2,y_1,y_2\in\mathcal{L}_2$.

The  Lurye system~(\ref{eq:Lurye}) is assumed to be well-posed with $G:\mathcal{L}_{2e}\rightarrow\mathcal{L}_{2e}$  linear time invariant (LTI) causal and stable, and with $\phi:\mathcal{L}_{2e}\rightarrow\mathcal{L}_{2e}$ memoryless and time-invariant.  With some abuse of notation we will use $G(s)$ to denote the transfer function corresponding to $G$. The nonlinearity $\phi$ is assumed to be montone in the sense that
$(\phi u) (t_1)\geq (\phi u)(t_2)$ for all $u(t_1)\geq  u(t_2)$. It is also assumed to be bounded in the sense that there exists a $C\geq 0$ such that $|(\phi u)(t)|\leq C|u(t)|$ for all $u(t)\in\mathbb{R}$. We say  $\phi$ is slope-restricted on $[0,k]$ if $0\leq (\phi u)(t_1) -(\phi u) (t_2))/(u(t_1)-u(t_2))\leq k$ for all $u(t_1)\neq u(t_2)$. We say $\phi$ is odd if $(\phi u)(t_1)=-(\phi u)(t_2)$ whenever $u(t_1)=-u(t_2)$.

\begin{definition}\label{def1}
Let $M:\mathcal{L}_{2}\rightarrow\mathcal{L}_{2}$  be LTI.
We say $M$ is a suitable multiplier for $G$ if there exists $\varepsilon>0$ such that
\begin{align}
\mbox{Re}\left \{
				M(j\omega) G(j\omega)
			\right \} > \varepsilon\mbox{ for all } \omega \in \mathbb{R}.
\end{align}
\end{definition}
\begin{remark}\label{rem_phase}
Suppose $M$ is a suitable multiplier for $G$ and $\angle G(j\omega) \leq -\pi/2 -\theta$ for some $\omega$ and $\theta$. Then $\angle M(j\omega) > \theta$. Similarly if $\angle G(j\omega) \geq \pi/2 +\theta$ then $\angle M(j\omega) < -\theta$.
\end{remark}

\begin{subtheorem}{definition}
\begin{definition}\label{def2a}
Let $\mathcal{M}$ be the class of LTI $M:\mathcal{L}_{2}\rightarrow\mathcal{L}_{2}$  whose implulse response is given by
\begin{equation}\label{def_m}
m(t) = m_0 \delta(t)-h(t)-\sum_{i=1}^{\infty}h_i \delta(t-t_i),
\end{equation}
with
\begin{equation}
\begin{split}
 h(t) & \geq 0 \mbox{ for all } t\mbox{, }h_i\geq 0 \mbox{ for all } i\\
& 
\mbox{and }
\| h\|_1+\sum_{i=1}^{\infty} h_i \leq m_0.
\end{split}
\end{equation}
\end{definition}
\begin{definition}\label{def2b}
Let $\mathcal{M}_{\mbox{odd}}$ be the class of LTI $M:\mathcal{L}_{2}\rightarrow\mathcal{L}_{2}$  whose implulse response is given by (\ref{def_m})
with
\begin{equation}
\| h\|_1+\sum_{i=1}^{\infty} |h_i| \leq m_0.
\end{equation}
\end{definition}
\end{subtheorem}
\begin{remark}
$\mathcal{M}\subset\mathcal{M}_{\mbox{odd}}$.
\end{remark}

The Lurye system (\ref{eq:Lurye}) is said to be absolutely stable for a particular $G$ if it is stable for all $\phi$ in some class $\Phi$. In particular, if there is a suitable $M\in\mathcal{M}$ for $G$ then it is absolutely stable for the class of memoryless time-invariant monotone bounded nonlinearities;  if there is a suitable $M\in\mathcal{M}_{\mbox{odd}}$ for $G$ then it is absolutely stable for the class of memoryless time-invariant odd monotone bounded nonlinearities. Furthermore, if there is a suitable $M\in\mathcal{M}$ for $1/k+G$ then it is absolutely stable for the class of memoryless time-invariant slope-restricted nonlinearities in $[0,k]$;  if there is a suitable $M\in\mathcal{M}_{\mbox{odd}}$ for $1/k+G$ then it is absolutely stable for the class of memoryless time-invariant odd slope-restricted nonlinearities \cite{Zames68,Carrasco:EJC}.

\subsection{Other notation}
Let $x = [y]_{[z,w]} $ denote $y$ modulo the interval $[z,w]$: i.e. the unique number $x\in[z,w)$ such that there is an integer $n$ with $y = x + n(w-z)$.

In our statement of results (i.e. Sections~\ref{Main},~\ref{Int} and~\ref{Exa}) phase is expressed in degrees. In the technical proofs (i.e. the Appendix) phase is expressed in radians.

\subsection{Duality approach}

The following result is similar in spirit to that in \cite{Jonsson:96} where a proof is sketched for the odd case. Both results can be derived from the duality theory of J\"{o}nsson \cite{Jonsson96thesis,Jonsson97,Jonsson99}; see \cite{Zhang:21} for the corresponding derivation in the discrete-time case. Nevertheless, several details are different. In particular, in  \cite{Jonsson:96}  only rational plants $G$ and rational multipliers $M$ are considered; this excludes both plants with delay and so-called ``delay multipliers.'' Expressing the results in terms of single parameter delay multipliers also gives insight. We exclude frequencies $\omega=0$ and $\omega\rightarrow\infty$; it is immediate that we must have  $\mbox{Re}\left \{M(0)G(0)\right \}\geq 0$; by contrast $M(\infty)$ need not be well-defined in our case.

\begin{definition}
Define the single parameter delay multipliers $M^-_\tau$ and $M^+_\tau$ as $M^-_\tau(s) = 1 -e^{-\tau s}$ and  $M^+_\tau(s) = 1 +e^{-\tau s}$ with $\tau\in\mathbb{R}\backslash 0$.
Let $\mathcal{M}^- \subset \mathcal{M}$ be the set $\mathcal{M}^- = \{M^-_{\tau} \,:\, \tau \in \mathbb{R} \backslash 0\}$.  Let 
$\mathcal{M}^+ \subset \mathcal{M_{\mbox{odd}}}$ 
be the set $\mathcal{M}^+ = \{M^+_\tau\,:\, \tau \in \mathbb{R}\backslash 0\}$.
\end{definition}

\begin{subtheorem}{theorem}
\begin{theorem}\label{Jthm_a}
Let $G$ be causal, LTI and stable. Assume there exist $0<\omega_1<\cdots<\omega_N<\infty$, and non-negative $\lambda_1, \lambda_2, \ldots, \lambda_N$, where $\sum_{r=1}^N\lambda_r>0$, such that
\begin{equation}\label{thm1_ineq}
\sum_{r=1}^N\lambda_r \mbox{Re}\left \{
				 M^-_{\tau}(j\omega_r) G(j\omega_r)
			\right \} \leq 0 \mbox{ for all }M^-_{\tau}\in\mathcal{M}^-.
\end{equation}
Then there is no suitable  $M\in\mathcal{M}$ for $G$.
\end{theorem}
\begin{theorem}\label{Jthm_b}
Let $G$ be causal, LTI and stable. Assume, in addition to the conditions of Theorem~\ref{Jthm_a}, that
\begin{equation}\label{thm1b_ineq}
\sum_{r=1}^N\lambda_r \mbox{Re}\left \{
				 M^+_{\tau}(j\omega_r) G(j\omega_r)
			\right \} \leq 0 \mbox{ for all }M^+_{\tau}\in\mathcal{M}^+.
\end{equation}
Then there is no suitable  $M\in\mathcal{M}_{\mbox{odd}}$ for $G$.
\end{theorem}
\end{subtheorem}

\begin{remark} The observation is made in \cite{Chang:12} that by the Stone-Weirstrass theorem it is sufficient to characterise $\mathcal{M}$ in terms of delay multipliers: i.e. as the class of LTI $M:\mathcal{L}_{2}\rightarrow\mathcal{L}_{2}$  whose impulse response is given by
\begin{equation}
m(t) = m_0 \delta(t)-\sum_{i=1}^{\infty}h_i \delta(t-t_i),
\end{equation}
with
\begin{equation}
h_i\geq 0 \mbox{ for all } i\mbox{ and }
\sum_{i=1}^{\infty} h_i \leq m_0.
\end{equation}
Similarly $\mathcal{M}_{\mbox{odd}}$ can be characterised as the  class of LTI $M:\mathcal{L}_{2}\rightarrow\mathcal{L}_{2}$  whose impulse response is given by
\begin{equation}
m(t) = m_0\delta(t) -\sum_{i=1}^{\infty}h_i \delta(t-t_i),
\end{equation}
with
\begin{equation}
\sum_{i=1}^{\infty} |h_i| \leq m_0.
\end{equation}
Such delay multipliers are excluded entirely from \cite{Jonsson:96}, but in this sense both Theorems~\ref{Jthm_a} and~\ref{Jthm_b} follow almost immediately.
\end{remark}


\subsection{Frequency interval approach}
In \cite{Wang:18} we presented the following phase limitation for the frequency intervals $[\alpha,\beta]$ and $[\gamma,\delta]$.

\begin{subtheorem}{theorem}
\begin{theorem}[\cite{Wang:18}]\label{Meg_a}
Let $0<\alpha<\beta<\gamma<\delta$ and define
\begin{equation}
\rho^c = \sup_{t>0}\frac{|\psi(t)|}{\phi(t)},
\end{equation}
with
\begin{equation}
\begin{split}
\psi(t) & = \frac{\lambda \cos (\alpha t)}{t}-\frac{\lambda \cos (\beta t)}{t}- \frac{\mu \cos (\gamma t)}{t}+\frac{\mu \cos (\delta t)}{t},\\
\phi(t) & = \lambda(\beta-\alpha)+\kappa\mu(\delta-\gamma)+\phi_1(t),\\
\phi_1(t) & = \frac{\lambda \sin (\alpha t)}{t}-\frac{\lambda \sin (\beta t)}{t}+ \frac{\kappa\mu \sin (\gamma t)}{t}-\frac{\kappa\mu \sin (\delta t)}{t},
\end{split}
\end{equation}
and with $\lambda>0$ and $\mu>0$ satisfying
\begin{equation}
\frac{\lambda}{\mu} = \frac{\delta^2-\gamma^2}{\beta^2-\alpha^2},
\end{equation}
and $\kappa>0$.
Let $M$ be an OZF multiplier and suppose
\begin{equation}\label{M_up}
\mbox{Im}(M(j\omega))>\rho\mbox{Re}(M(j\omega))\mbox{ for all } \omega\in[\alpha,\beta],
\end{equation}
and
\begin{equation}\label{M_dn}
\mbox{Im}(M(j\omega))<-\kappa\rho\mbox{Re}(M(j\omega))\mbox{ for all } \omega\in[\gamma,\delta],
\end{equation}
for some $\rho>0$. Then $\rho<\rho^c$ if $M\in\mathcal{M}$.

The result also holds if we replace (\ref{M_up}) and (\ref{M_dn}) with
\begin{equation}
\mbox{Im}(M(j\omega))<-\rho\mbox{Re}(M(j\omega))\mbox{ for all } \omega\in[\alpha,\beta],
\end{equation}
and
\begin{equation}
\mbox{Im}(M(j\omega))>\kappa\rho\mbox{Re}(M(j\omega))\mbox{ for all } \omega\in[\gamma,\delta].
\end{equation}
\end{theorem}
\begin{theorem}[\cite{Wang:18}]\label{Meg_b}
Suppose, in addition to the conditions of Theorem~\ref{Meg_a}, that
\begin{equation}
\rho^c_{\mbox{odd}}  = \sup_{t>0}\frac{|\psi(t)|}{\tilde{\phi}(t)},
\end{equation}
with
\begin{equation}
\tilde{\phi}(t)  =  \lambda(\beta-\alpha)+\kappa\mu(\delta-\gamma)-|\phi_1(t)|.
\end{equation}
Then  $\rho<\rho^c_{\mbox{odd}}$ if $M\in\mathcal{M}_{\mbox{odd}}$.
\end{theorem}
\end{subtheorem}

%% file: Part3_PhaseLimitations01.tex
\section{Main results: duality approach}\label{Main}

Applying Theorem~\ref{Jthm_a} or~\ref{Jthm_b} with $N=1$ yields no significant result beyond the trivial statement that if $\mbox{Re}[G(j\omega)]<0$ and $\mbox{Im}[G(j\omega)]=0$ at any $\omega$ then there can be no suitable multiplier. This is in contrast with the discrete-time case where there are non-trivial phase limitations at single frequencies \cite{Zhang:21}. 

Even with $N=2$, it is not straightforward to apply Theorems~\ref{Jthm_a} or~\ref{Jthm_b} directly, as they require an optimization at each pair of frequencies. Nevertheless, setting $N=2$ yields the following phase limitations:

\begin{subtheorem}{theorem}
\begin{theorem}\label{thm:2a}
Let $a, b \in \mathbb{Z}^+$ and  let $G$ be causal, LTI and stable. If there exists $\omega_0\in\mathbb{R}$ such that
 \begin{align}\label{G_ineq}
\left |
\frac{
b\angle G(aj\omega_0 ) - a \angle G(bj\omega_0)
}
{a+b-p}
\right |>  180^o,
\end{align}
with $p=1$
 then there is no suitable $M\in\mathcal{M}$ for $G$.
\end{theorem}
\begin{theorem}\label{thm:2b}
Let $a, b \in \mathbb{Z}^+$ and  let $G$  be causal, LTI and stable.
 If there exists $\omega_0\in\mathbb{R}$ such that (\ref{G_ineq}) holds
where $p=1$ when both $a$ and $b$ are odd but $p=1/2$ if either $a$ or $b$ are even,
 then there is no suitable $M\in\mathcal{M}_{\mbox{odd}}$ for $G$.
\end{theorem}
\end{subtheorem}


Figs~\ref{test02_fig1} and~\ref{test02_fig2} illustrate Theorems~\ref{thm:2a} and~\ref{thm:2b} respectively for the specific case that $\angle{G( j\omega_a}) > 170^o$ for some frequency $\omega_a$. The results put limitations on the phase of $G$ at frequencies that are rational multiples of $\omega_a$ (i.e. at $b\omega_0$ where $\omega_a=a\omega_0$ and where $a$ and $b$ are coprime integers).

\begin{figure}[htbp]
\begin{center}
\includegraphics[width = 0.9\linewidth]{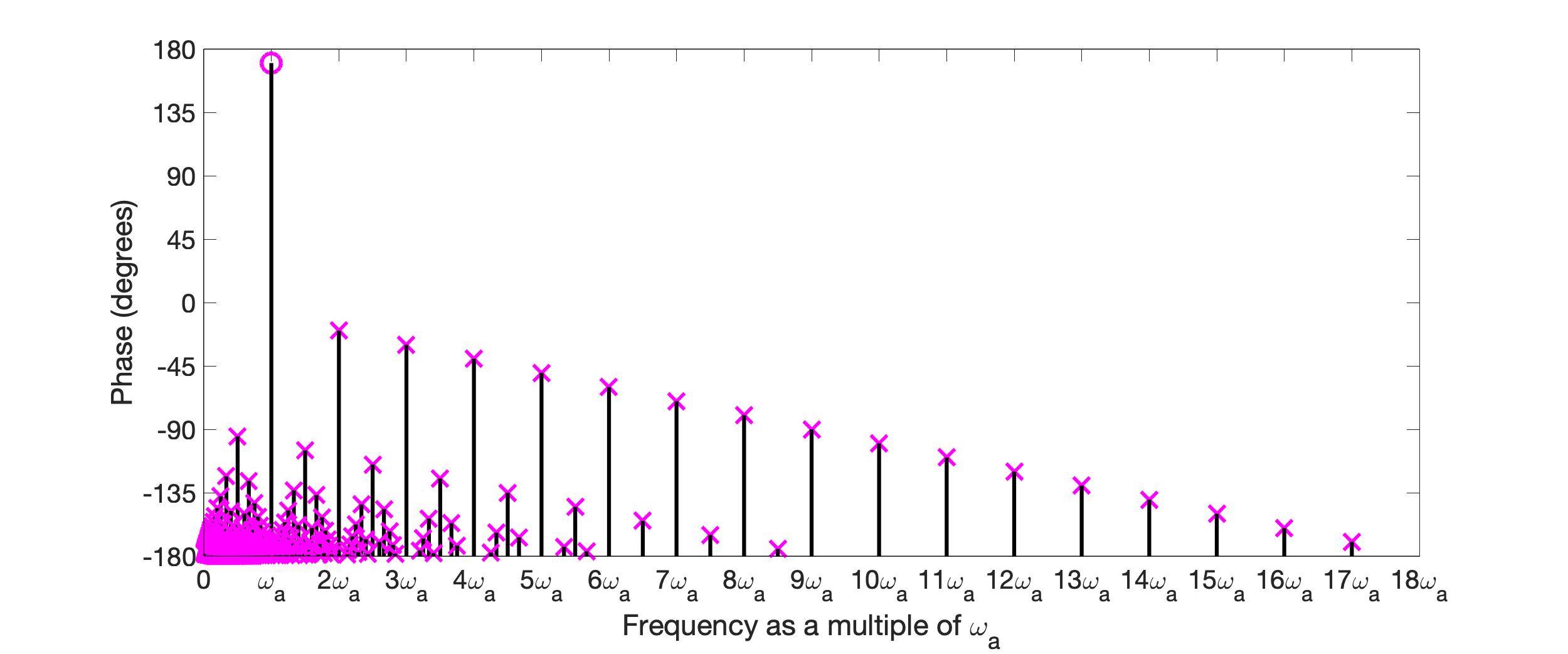}
\caption{Forbidden regions for the phase of $G(j\omega)$ when the phase at some $\omega_a$ is greater than $170^o$. }\label{test02_fig1}
\end{center}
\end{figure}
\begin{figure}[htbp]
\begin{center}
\includegraphics[width = 0.9\linewidth]{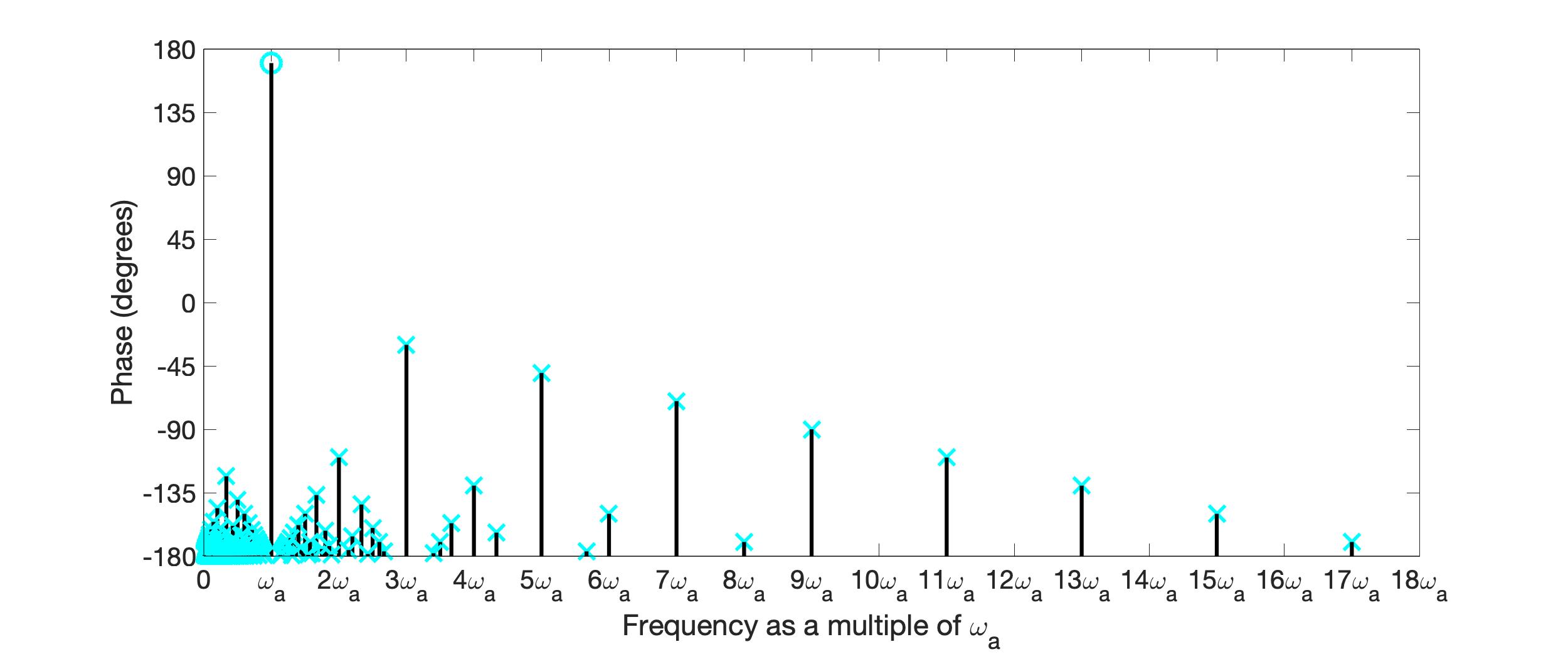}
\caption{Forbidden regions for the phase of $G(j\omega)$ when the phase at some $\omega_a$ is greater than $170^o$ (odd nonlinearity). }\label{test02_fig2}
\end{center}
\end{figure}


The results may also be expressed as phase limitations on the multipliers themselves. Counterparts to Theorems~\ref{thm:2a} and~\ref{thm:2b} follow  as corollaries and are equivalent results.
\begin{subtheorem}{corollary}
\begin{corollary}\label{cor:2a}
Let $a, b \in \mathbb{Z}^+$ and  let $M\in\mathcal{M}$. Then
\begin{equation}\label{cor_ineq}
\left |\frac{b\angle M(aj\omega )-a\angle M(bj\omega )}{a/2+b/2-p}\right |  \leq  180^o,
\end{equation}
for all $\omega\in\mathbb{R}$ with $p=1$.
\end{corollary}
\begin{corollary}\label{cor:2b}
Let $a, b \in \mathbb{Z}^+$ and  let $M\in\mathcal{M}_{\mbox{odd}}$. Then inequality (\ref{cor_ineq}) holds
 for all $\omega\in\mathbb{R}$ where $p=1$ when both $a$ and $b$ are odd but $p=1/2$ if either $a$ or $b$ are even.
\end{corollary}
\begin{proof}
Immediate: see Remark~\ref{rem_phase}.
\end{proof}
\end{subtheorem}

Figs~\ref{test04_fig1} and~\ref{test04_fig2} are the counterparts to Figs~\ref{test02_fig1} and~\ref{test02_fig2} (if the phase of $G$ is greater than $170^o$ at some $\omega_a$ then any suitable multiplier $M$ must have phase less than $-80^o$ at $\omega_a$).
Corollaries~\ref{cor:2a} and~\ref{cor:2b}  can also be visualised for specific values of $a$ and $b$ with plots of the phase of $M(bj\omega_0 )$ against the phase of  $M(aj\omega_0 )$ as $\omega_0$ varies:  see Figs~\ref{Fig03_1} to~\ref{Fig03_3}. 
Fig~\ref{Fig03_1} also shows boundary points parameterised by $\kappa$ which is associated with the frequency interval apprach and discussed in Section~\ref{Int}.

\begin{figure}[htbp]
\begin{center}
\includegraphics[width = 0.9\linewidth]{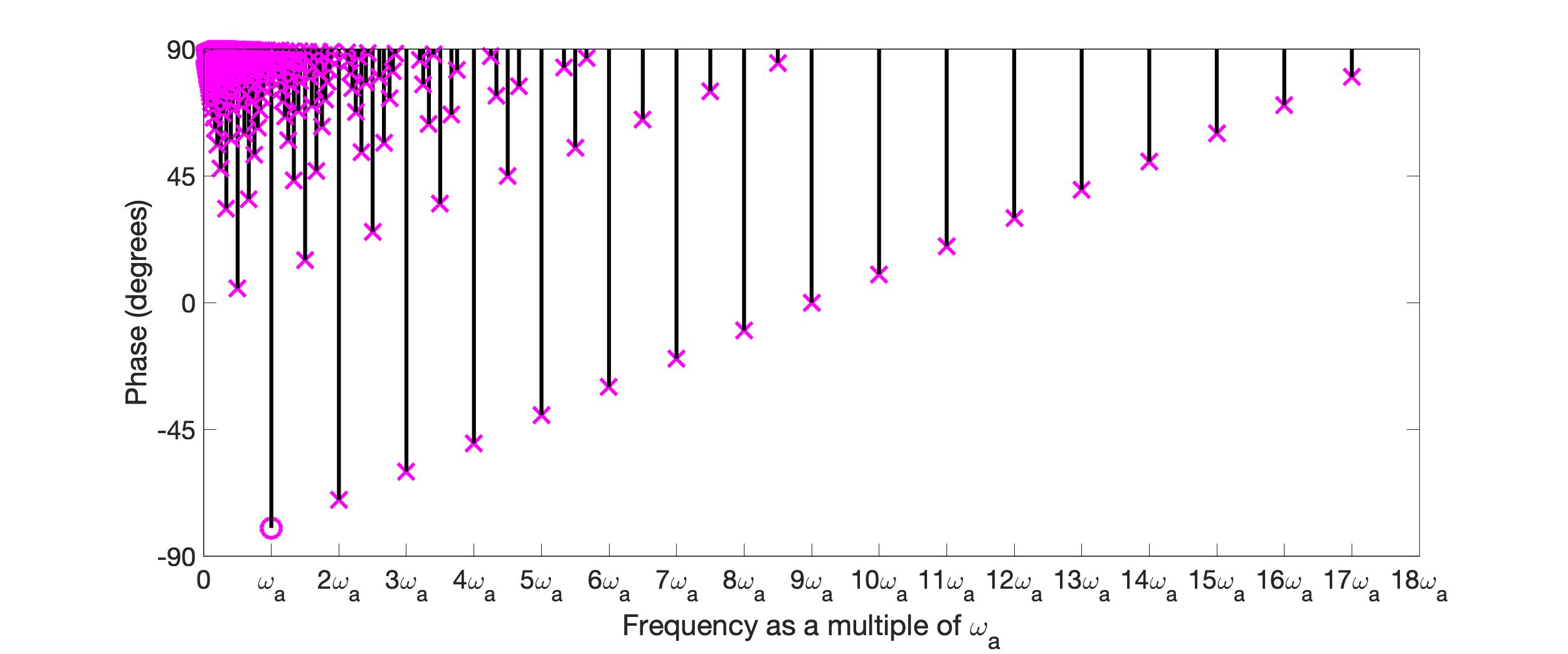}
\caption{Forbidden regions for the phase of $M\in\mathcal{M}$ when the phase at some $\omega_a$ is less than $-80^o$. }\label{test04_fig1}
\end{center}
\end{figure}
\begin{figure}[htbp]
\begin{center}
\includegraphics[width = 0.9\linewidth]{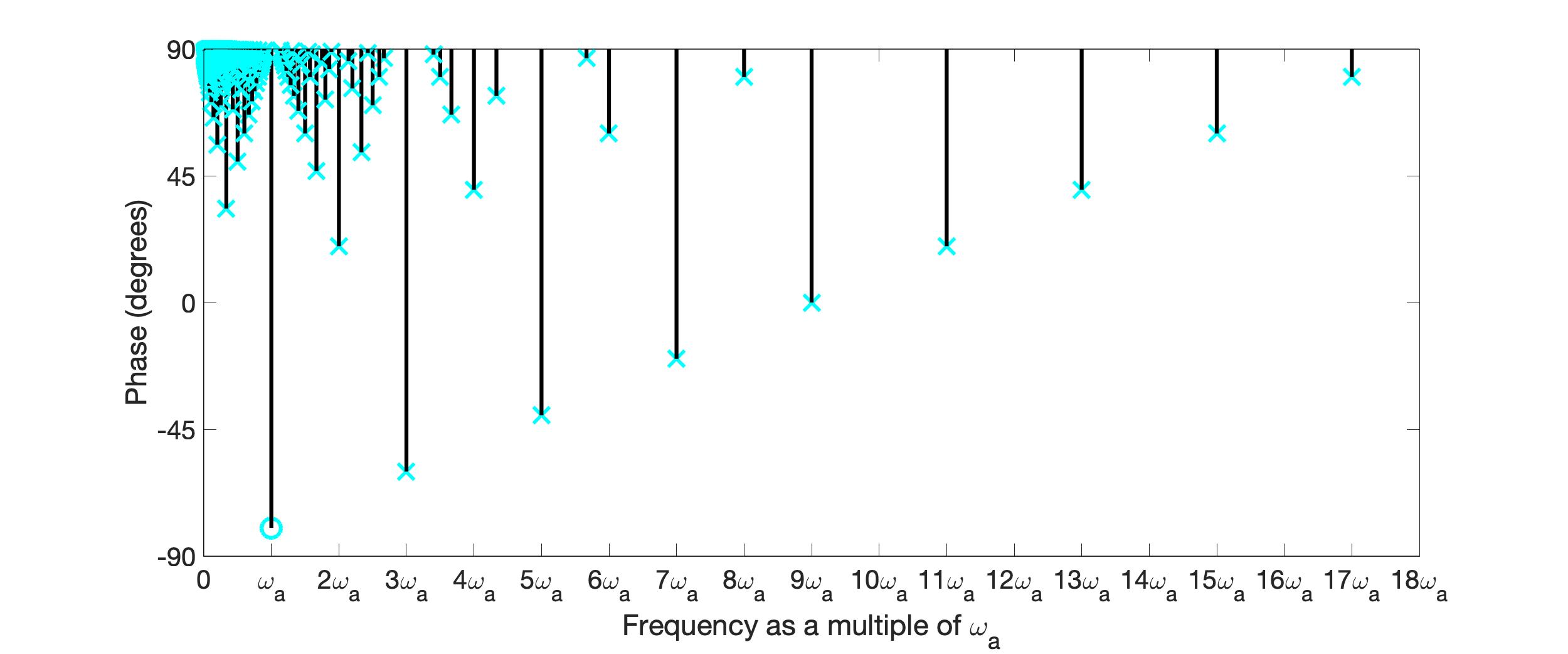}
\caption{Forbidden regions for the phase of $M\in\mathcal{M}_{\mbox{odd}}$ when the phase at some $\omega_a$ is less than $-80^o$. }\label{test04_fig2}
\end{center}
\end{figure}

\begin{figure}[htbp]
\begin{center}
\includegraphics[width = 0.9\linewidth]{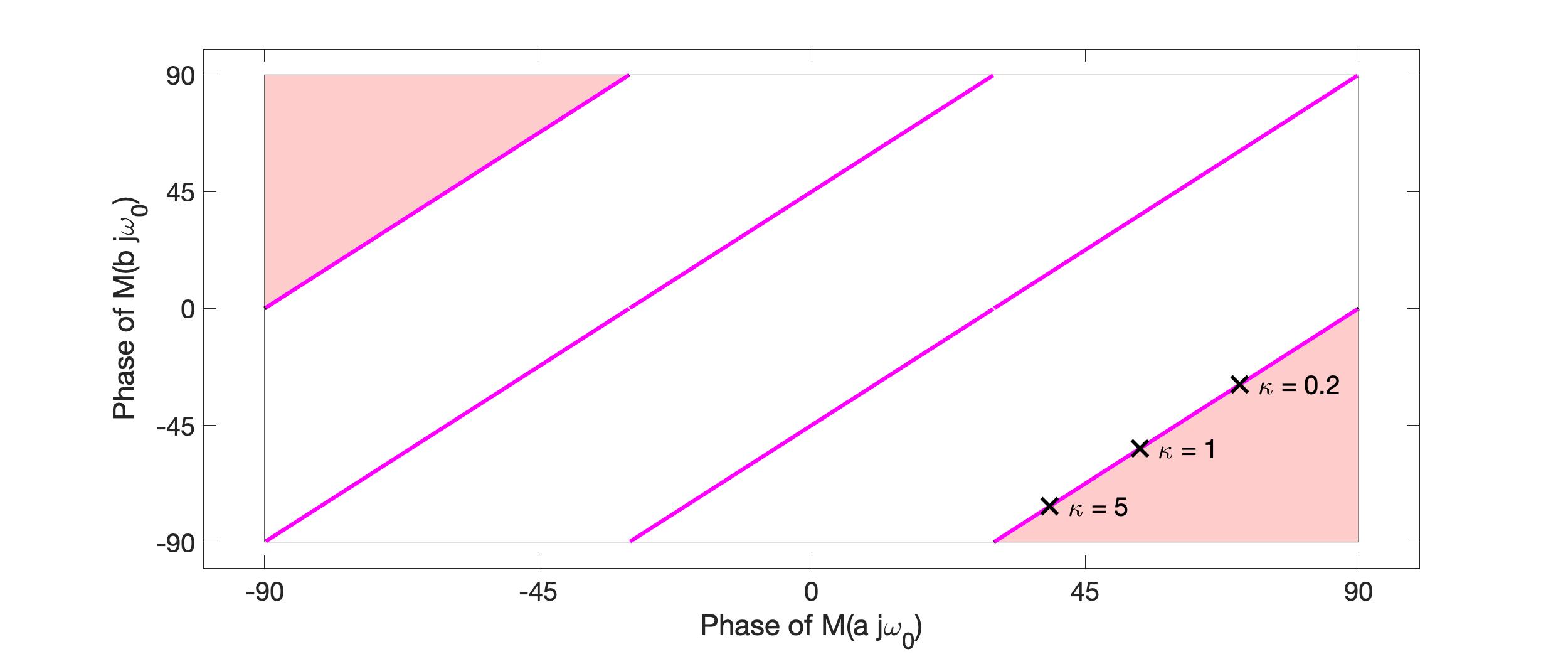}
\caption{Phase vs phase plot illustrating Corollary~\ref{cor:2a} with $a=2$, $b=3$. If $M\in\mathcal{M}$ then the pink regions are forbidden. The phase vs phase plots of elements of $\mathcal{M}^-$ are shown in magenta.
Also shown are the points $(\arctan \rho^c,-\arctan \kappa \rho^c)$ when $a=2$ and $b=3$, when $\kappa$ takes the values $0.2$, $1$ and $5$ and when $\rho^c$ is defined as in Corollary~\ref{m_corollary_a}.
}\label{Fig03_1}
\end{center}
\end{figure}

\begin{figure}[htbp]
\begin{center}
\includegraphics[width = 0.9\linewidth]{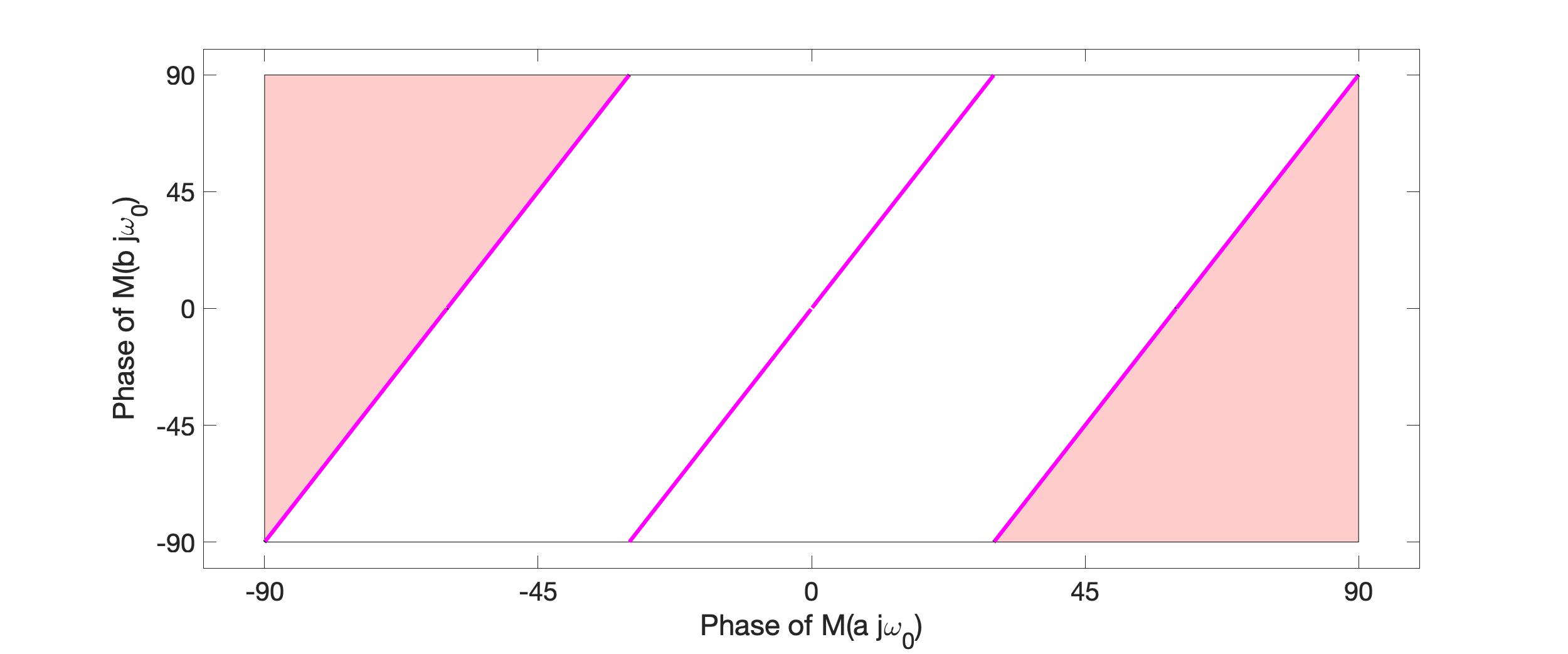}
\caption{Phase vs phase plot illustrating both Corollaries~\ref{cor:2a} and~\ref{cor:2b} with $a=1$, $b=3$. If $M\in\mathcal{M}$ or $M\in\mathcal{M}_{\mbox{odd}}$  then the pink regions are forbidden. The phase vs phase plots of elements of $\mathcal{M}^-$ and $\mathcal{M}^+$ coincide and are shown in magenta.}\label{Fig03_2}
\end{center}
\end{figure}

\begin{figure}[htbp]
\begin{center}
\includegraphics[width = 0.9\linewidth]{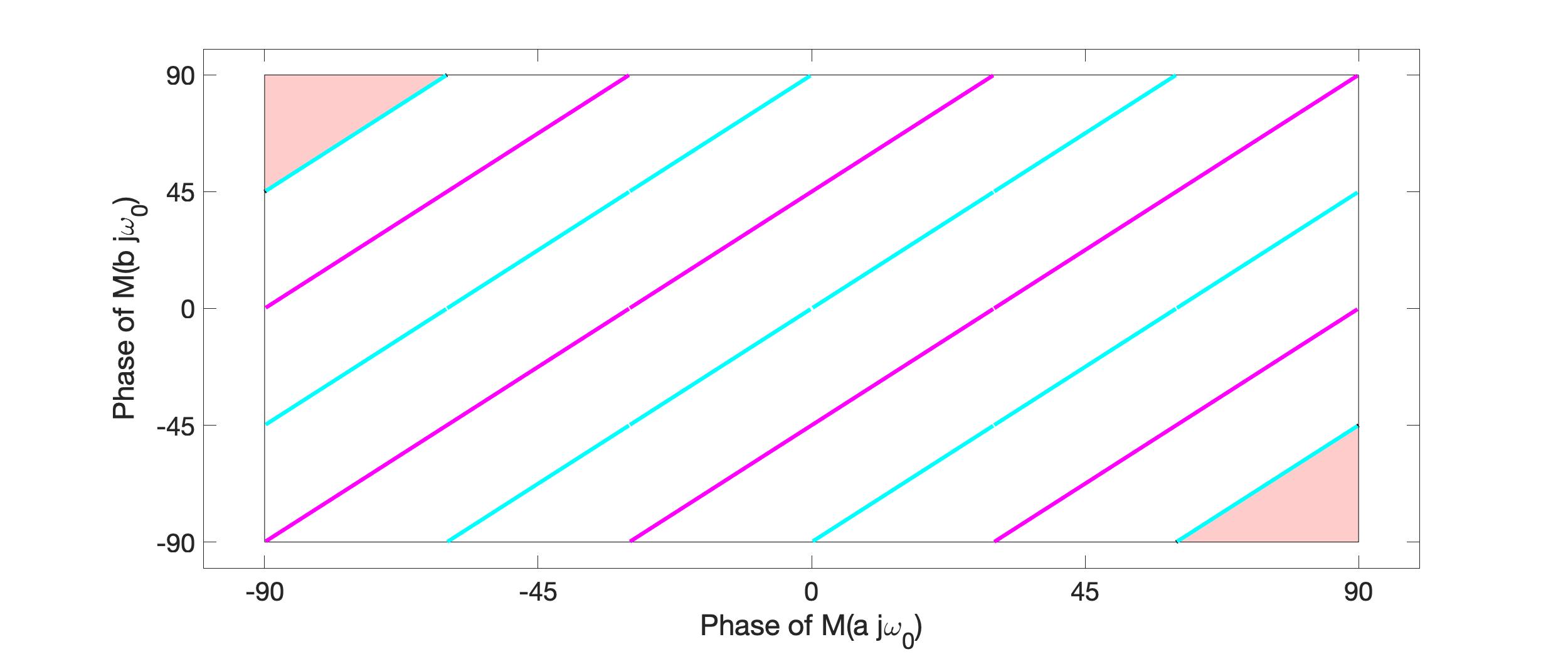}
\caption{Phase vs phase plot illustrating Corollary~\ref{cor:2b} with $a=2$, $b=3$. If $M\in\mathcal{M}_{\mbox{odd}}$  then the pink regions are forbidden. The phase vs phase plots of elements of $\mathcal{M}^-$ are shown in magenta (compare Fig~\ref{Fig03_1}) while the phase vs phase plots of elements of $\mathcal{M}^+$ are shown in cyan.}\label{Fig03_3}
\end{center}
\end{figure}

The bounds are tight in the sense that if $a$ and $b$ are coprime then there exist (many) $M_{\tau}^-\in\mathcal{M}^-$  such that  $b\angle M_{\tau}^-(a j\omega_0 )-a\angle M_{\tau}^-(b j\omega_0) = (a/2+b/2-1)180^o$. Specifically this holds for any $\tau$ that satisfies $[a\tau/\omega_0]_{[0,2\pi]}>2\pi-2\pi/b$ and $[b\tau/\omega_0]_{[0,2\pi]} < 2\pi/a$.
Similarly if $a$ and $b$ are coprime and either $a$ or $b$ are even there exist (many)  $M_{\tau}^+\in\mathcal{M}^+$ such that $b\angle M_{\tau}^+(a j\omega_0)-a\angle M_{\tau}^+(b j\omega_0 ) = (a/2+b/2-1/2)180^o$. Specifically this holds for any $\tau$  that satisfies $\pi-\pi/b <[a\tau/\omega_0]_{[0,2\pi]}<\pi$ and  $\pi<[b\tau/\omega_0]_{[0,2\pi]}<\pi+\pi/a$.




In the examples below the phases of the  objects $G(a j\omega )$ and $G(b j\omega )$  are computed separately. They should each  have phase on the interval $(-180^o, 180^o)$ and so may be easily computed without the possibility of phase wrapping ambiguity at local points or over local regions. Provided the transfer functions are sufficiently smooth they can be computed accurately. Nevertheless, it is possible to write (\ref{G_ineq})  in terms of a single transfer function since
\begin{equation}
	b\angle G(a j\omega )-a\angle G(b j\omega ) = \angle \bar{G}_{a,b}(j\omega)
\end{equation}
where
\begin{equation}
\bar{G}_{a,b}(s) = \frac{G( a s)^b}{G( b s)^a}.
\end{equation}
It thus requires, for given values of $a$ and $b$, the computation of the maximum (or minimum) phase of a single transfer function. In this sense the computational requirement is comparable to that of the off-axis circle criterion \cite{Cho:68}, a classical tool.


It may also be necessary to compute the criterion for several positive integer values of $a$ and $b$. The number of different values is finite and can be bounded. 
Suppose the maximum phase of $G$ is $180^o-\phi_{\min}$ and the minimum phase is $-180^o+\theta_{\max}$, where $\phi_{\min}>0, \theta_{\max}>0$. Then $a \theta_{\max} +b\phi_{\min} < p\times 180^o$. So it is sufficient to choose (say) all $a<p/\theta_{\max} \times 180^o$ and $b<p/\phi_{\min} \times 180^o$ which yields a finite set of values.

%% file: Part3a_Intervals.tex
\section{Relation to the frequency interval approach}\label{Int}

Corollaries \ref{cor:2a} and \ref{cor:2b} may be interpreted as saying that given an upper (or lower) threshold on the phase of a suitable multiplier $M$ at frequency $a\omega_0$ there is a lower (or upper) threshold on the phase on $M$ at frequency $b\omega$. It is natural to compare this with the frequency interval approach, where an upper (or lower) threshold on the phase of $M$ over an interval $[\alpha,\beta]$ implies a lower (or upper) threshold on the phase of $M$ over the interval $[\gamma,\delta]$.

Let us begin by considering Theorems~\ref{Meg_a} and~\ref{Meg_b}  in the limit as the length of the intervals becomes zero. We obtain the following corollaries. The results requires the ratio of the limiting frequencies to be rational.

\begin{subtheorem}{corollary}
\begin{corollary}\label{m_corollary_a}
{ \everymath={\displaystyle}
For $t>0$, define
\begin{equation}
q_-(t)  = \left \{
			\begin{array}{l}
				 \frac{b\sin (a t)-a\sin (b t)}{ b+\kappa a- b \cos (a  t)-\kappa a\cos(b  t)}
						\mbox{ for }[t]_{[0,\pi]}\neq 0,\\
\\
			0 \mbox{ for } [t]_{[0,\pi]}=0,
			\end{array}
		\right .\\
\end{equation}
where $a$ and $b$ are coprime and $\kappa>0$. }Define also
\begin{equation}
\overline{\rho}^c  = \sup_{t>0}  |q_-(t)|.
\end{equation}
Let $M$ be an OZF multiplier and suppose
\begin{equation}
\mbox{Im}(M(aj\omega_0)>\rho\mbox{Re}(M(aj\omega_0)),
\end{equation}
and
\begin{equation}
\mbox{Im}(M(bj\omega_0)<-\kappa\rho\mbox{Re}(M(bj\omega_0)),
\end{equation}
for some $\omega_0>0$ and  $\rho>0$. Then $\rho<\rho^c$ if $M\in\mathcal{M}$.

\end{corollary}
\begin{corollary}\label{m_corollary_b}
{ \everymath={\displaystyle}
In addition to the conditions of Corollary~\ref{m_corollary_a}, define
\begin{equation}
q_+(t)  = \left \{
			\begin{array}{l}
 \frac{b\sin (a t)-a\sin (b t)}{ b+\kappa a+b \cos (a  t)+\kappa a\cos(b  t)}
						\mbox{ for }[t]_{[0,\pi]}\neq 0,\\
\\
			0 \mbox{ for } [t]_{[0,\pi]}=0,
			\end{array}
		\right .
\end{equation}
and
\begin{equation}
\overline{\rho}^c_{\mbox{odd}}  = \max\left (\sup_{t>0}|q_-(t) |,\sup_{t>0}|q_+(t)|\right ) .
\end{equation}
Then $\rho<\overline{\rho}^c$ if $M\in\mathcal{M}_{\mbox{odd}}$.
}
\end{corollary}
\end{subtheorem}

\begin{remark}
Equivalently, we can say if $\angle M(aj\omega_0)>\arctan \rho$ and  $\angle M(b j \omega_0)<-\arctan \kappa \rho$ then $\rho <\rho^c$ if $M\in\mathcal{M}$ and $\rho <\rho^c_{\mbox{odd}}$ if $M\in\mathcal{M}_{\mbox{odd}}$.
\end{remark}

It turns out that this is equivalent to the phase condition derived via the duality approach. The inequality boundaries  $\angle M(aj\omega_0)=\arctan \rho^c$ and $\angle M(b j \omega_0)=-\arctan \kappa \rho^c)$ (or $\angle M(aj\omega_0)=\arctan \rho^c_{\mbox{odd}}$ and $\angle M(b j\omega_0)=-\arctan \kappa \rho^c_{\mbox{odd}}$) are the same as those for Corollary~\ref{cor:2a} (or~\ref{cor:2b}), as illustrated in Fig~\ref{Fig03_1}.
Specifically we may say:
\begin{subtheorem}{theorem}
\begin{theorem}\label{Meg_equiv_a}
Corollary~\ref{m_corollary_a} and Theorem~\ref{thm:2a} are equivalent results. 
\end{theorem}
\begin{theorem}\label{Meg_equiv_b}
Corollary~\ref{m_corollary_b} and Theorem~\ref{thm:2b} are equivalent results.
\end{theorem}
\end{subtheorem}

%% file: Part4_Examples01.tex
\section{Examples}\label{Exa}

We demonstrate the new condition with three separate examples. In Examples~1 and~2 below we test the criterion for a finite number of coprime integers $a$ and $b$, and for all $\omega>0$; we also search over the slope restriction~$k$. We  run a bisection algorithm for~$k$ and,  for each candidate value of $k$, $a$ and~$b$, check whether the condition is satisfied for any $\omega>0$. Provided the phase of $1/k+G$ is sufficiently smooth, this can be implemented efficiently and systematically, for example by gridding $\omega$ sufficiently finely. There are several possible ways to reorder the computation.

\subsection{Example 1}

J\"{o}nsson and Laiou \cite{Jonsson:96} consider the plant
\begin{equation}\label{JL_G}
G(s) = \frac{s^2}{(s^2+\alpha)(s^2+\beta)+10^{-4}(14s^3+21s)},
\end{equation}
with $\alpha=0.9997$ and $\beta=9.0039$ and with positive feedback. 
They show that the rational multliper
\begin{equation}\label{JL_M}
M(s) = 1 - \left (\frac{2.5}{s+2.5}
		\right )^2.
\end{equation}
is suitable for $1/k-G(s)$ when $k=0.0048$.
Figure \ref{test05a_fig1} shows the phase of $M(j\omega)(1/k-G(j\omega))$ when $k=0.0048$. It can be seen to lie on the interval $[-90^o,90^o]$. They  also show no rational multiplier in $\mathcal{M}_{\mbox{odd}}$ exists when $k=0.0061$ by applying their criterion  with  $N=2$ and the choice $\omega_1=1$ and $\omega_2=3$.  
Fig \ref{test05a_fig2} shows $(3\angle(1/k-G(j\omega))-\angle(1/k-G(3j\omega)))/3$ when $k=0.0061$. It can be seen that the value drops below $-180^o$ near $\omega=1$. Thus Theorem~\ref{thm:2a} confirms there is no suitable multipler in either $\mathcal{M}$ or $\mathcal{M}_{\mbox{odd}}$.

J\"{o}nsson and Laiou \cite{Jonsson:96}  state `the choice of frequencies [...] is a delicate task.''' But a simple line search shows that there is an $\omega$ such that
$(3\angle(1/k-G(j\omega))-\angle(1/k-G(3j\omega)))/3<-180^o$ when $k = 0.0058926$ (see Fig \ref{test05a_fig3}) but $(3\angle(1/k-G(j\omega))-\angle(1/k-G(3j\omega)))/3>-180^o$ for all $\omega$ when  $k = 0.0058925$.
 By Theorem~\ref{thm:2a} there is no multiplier when $k=0.0058926$. By contrast, for this case the choice 
\begin{equation}\label{WPH_M}
M(s)=1-0.99999e^{-0.93287s}
\end{equation}
 is a suitable multiplier when $k=0.0058924$ (Fig \ref{test05a_fig4}). The various computed slopes $k$ are set out in Table~\ref{table_ex1}.

\begin{figure}[htbp]
\begin{center}
\includegraphics[width = 0.9\linewidth]{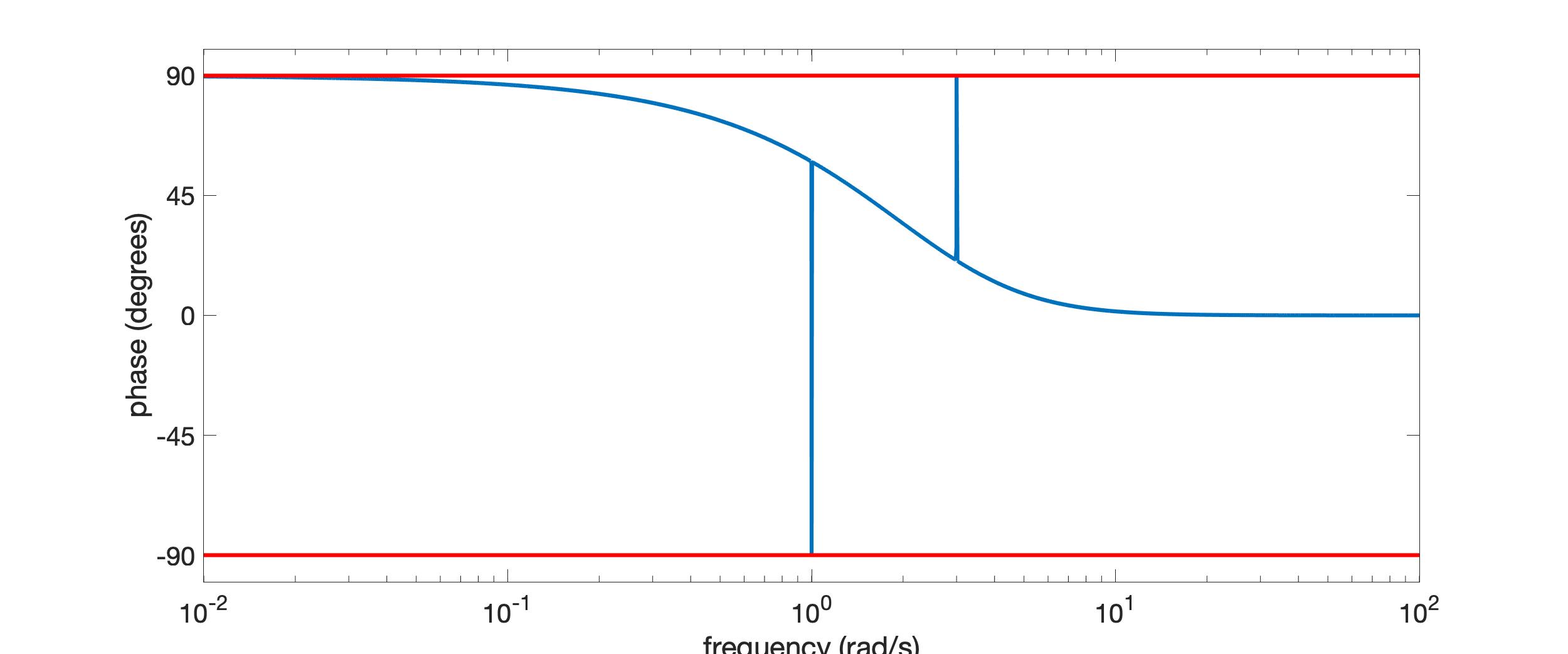}
\caption{Example 1. Phase of $M(j\omega)(1/k-G(j\omega))$ when $k=0.0048$ when $G$ is given by (\ref{JL_G}) and $M$ by (\ref{JL_M}). The phase lies on the interval $[-90^o,90^o]$ so this choice of $M$ is a suitable multiplier for $1/k-G$.}\label{test05a_fig1}
\end{center}
\end{figure}

\begin{figure}[htbp]
\begin{center}
\includegraphics[width = 0.9\linewidth]{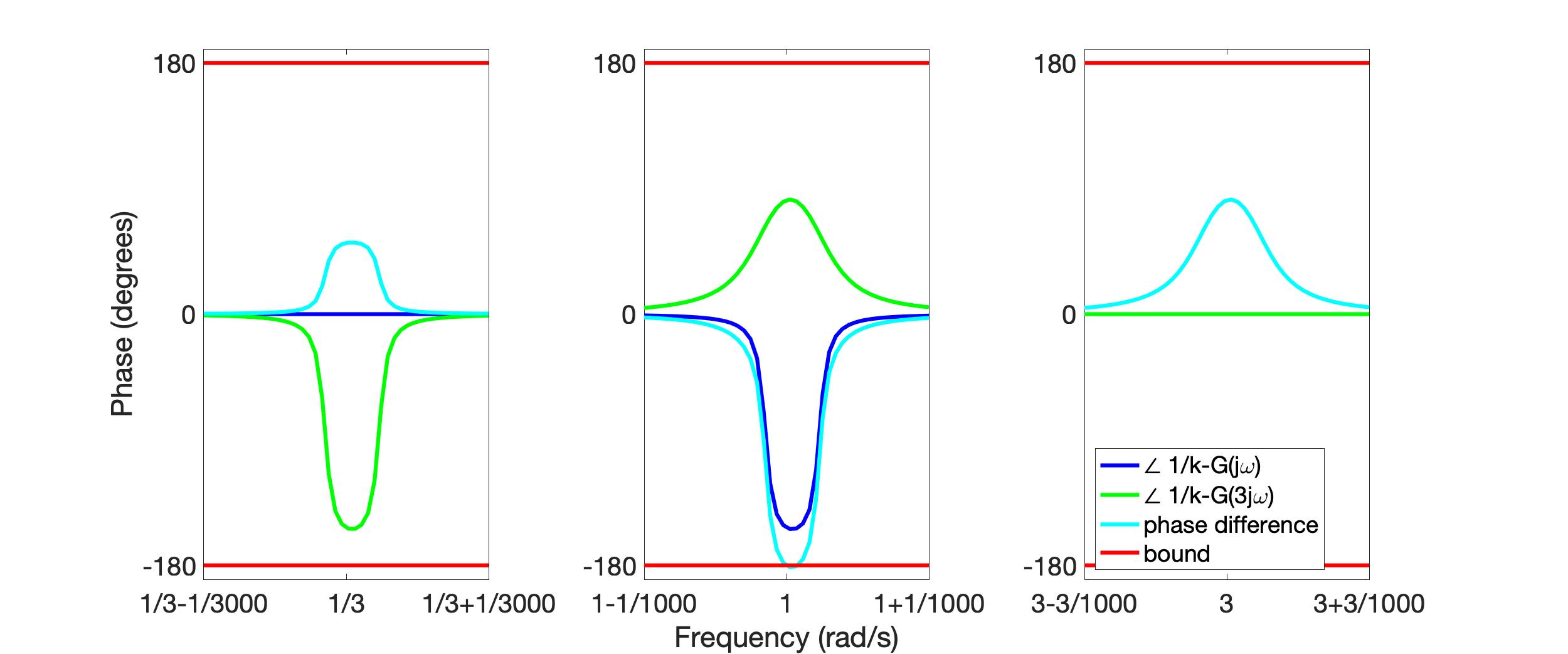}
\caption{Example 1. The phase difference $(3\angle(1/k-G(j\omega))-\angle(1/k-G(3j\omega)))/3$ when $G$ is given by (\ref{JL_G}) with $k=0.0061$. The value drops below $-180^o$ so by Theorem~\ref{thm:2a} there is no suitable multiplier.}\label{test05a_fig2}
\end{center}
\end{figure}

\begin{figure}[htbp]
\begin{center}
\includegraphics[width = 0.9\linewidth]{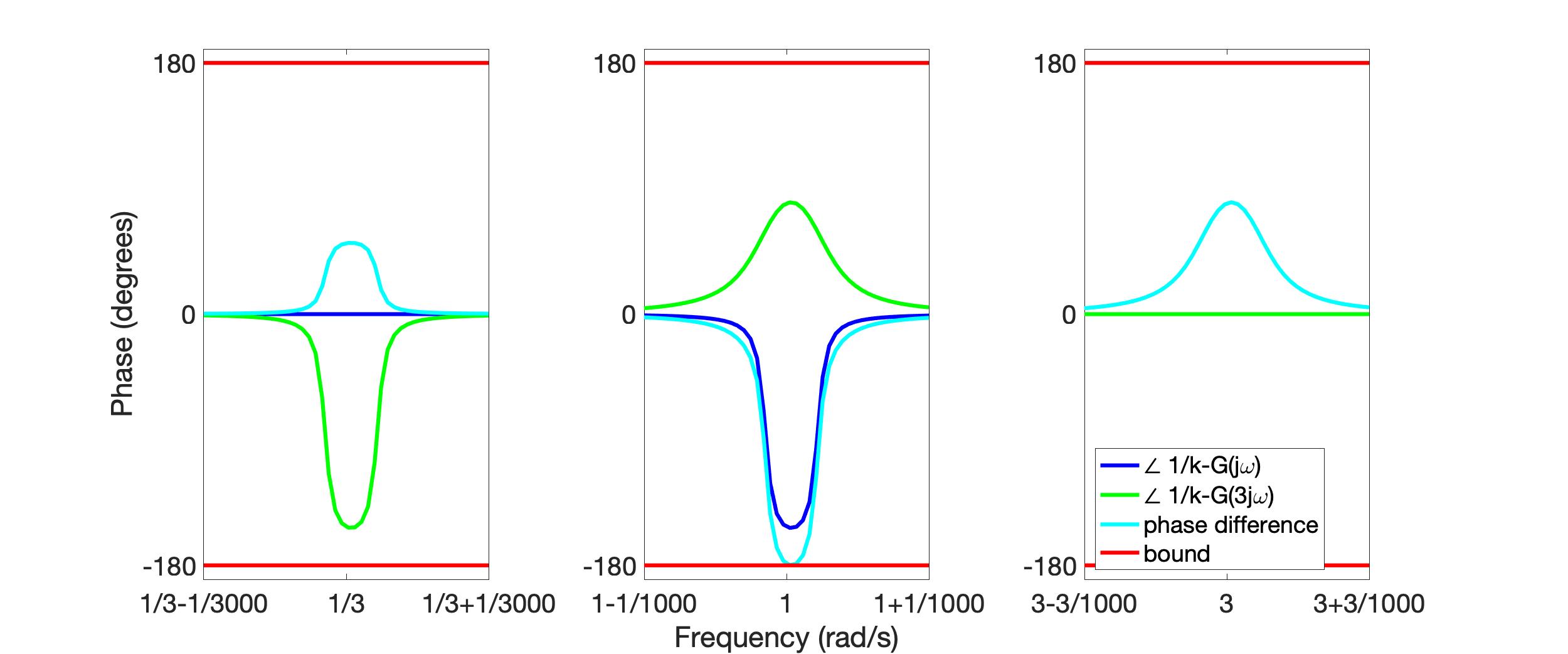}
\caption{Example 1. The phase difference $(3\angle(1/k-G(j\omega))-\angle(1/k-G(3j\omega)))/3$ when $G$ is given by (\ref{JL_G}) with $k=0.0058926$. The value drops below $-180^o$ so by Theorem~\ref{thm:2a} there is no suitable multiplier.}\label{test05a_fig3}
\end{center}
\end{figure}

\begin{figure}[htbp]
\begin{center}
\includegraphics[width = 0.9\linewidth]{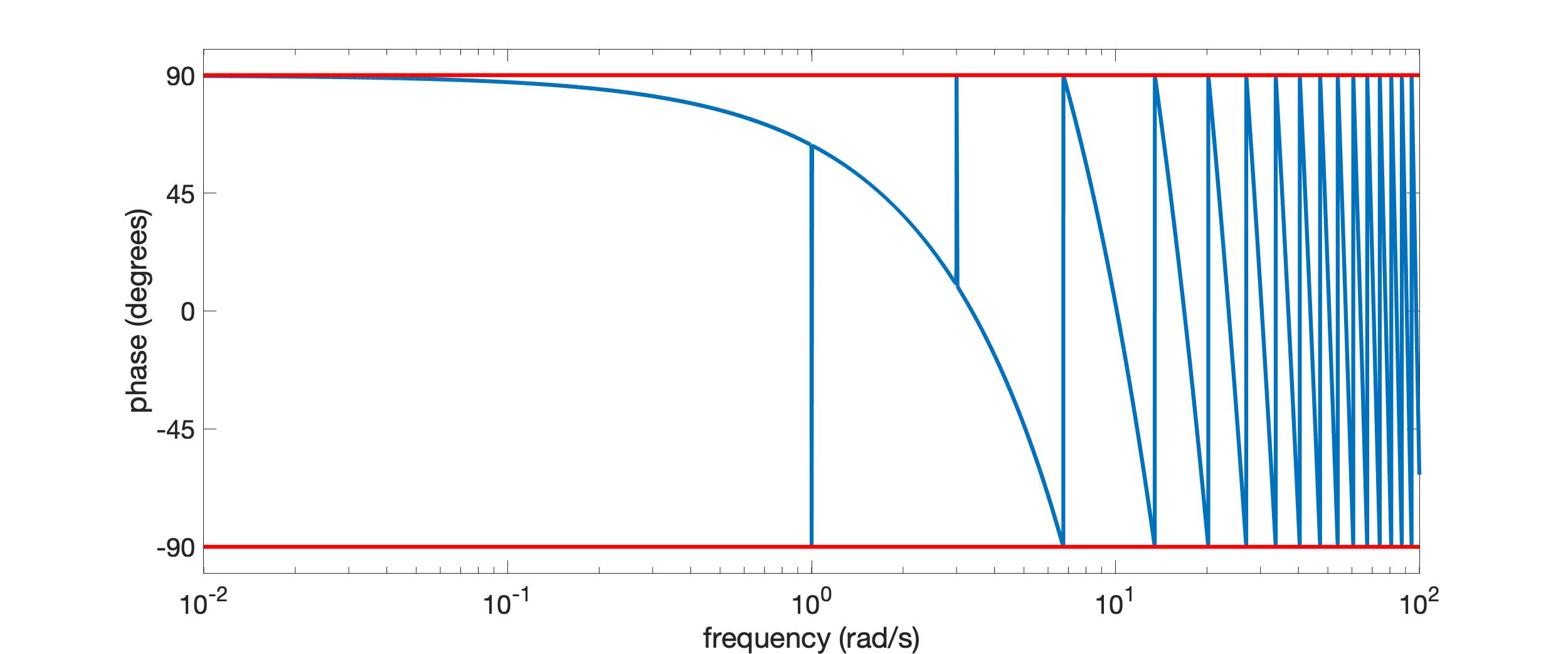}
\caption{Example 1. Phase of $M(j\omega)(1/k-G(j\omega))$ when $k=0.0058924$ when $G$ is given by (\ref{JL_G}) and $M$ by (\ref{WPH_M}). The phase lies on the interval $[-90^o,90^o]$ so this choice of $M$ is a suitable multiplier for $1/k-G$.}\label{test05a_fig4}
\end{center}
\end{figure}

\begin{table}
\begin{center}
    \begin{tabular}{ | l | c | c |}
    \hline
	&  \cite{Jonsson:96} & This paper\\
	\hline
    Slope $k$ for which a multiplier & & \\
 is  found & 0.0048 & 0.0058924\\
	\hline
    Slope $k$ for which there is  & & \\
guaranteed to be no multiplier & 0.0061 & 0.0058926\\
    \hline
    \end{tabular}
\caption{Various slopes for Example 1}\label{table_ex1}
\end{center}
\end{table}

\subsection{Example 2}
Consider the plant
\[G(s) = \frac{s^2}{(s^2+2\xi s + 1)^2}\mbox{ with }\xi>0.\]
O'Shea \cite{OShea67} shows that there is a suitable multiplier in $\mathcal{M}$ for $1/k+G$ when $\xi>1/2$ and $k>0$. By contrast in \cite{Wang:18} we showed that there is no suitable multiplier in $\mathcal{M}$ when $\xi=0.25$ and $k$ is sufficiently large. Specifically the phase of $G(j\omega)$ is above $177.98^o$ on the interval $\omega\in [0.02249,0.03511]$ and below $-177.98^o$ on the interval $\omega\in [1/0.03511,1/0.02249]$. A line search yields that the same condition is true for the phase of $1/k+G(j\omega)$ with $k\geq 269,336.3$ (see Fig~\ref{test19e_fig1}). Hence there is no suitable multipler $M\in\mathcal{M}$ for $1/k+G$ with $k\geq 269,336.3$.

By contrast, Theorem~\ref{thm:2a} with $a=4$ and $b=1$ yields  there is no suitable multipler $M\in\mathcal{M}$ for $1/k+G$ with $k\geq 32.61$. Specifically the phase $(4\angle (1/k+G(j\omega))-\angle(1/k+G(4j\omega)))/4$ exceeds $180^o$ when $k\geq 32.61$ (see Figs~\ref{test19e_fig3} and~\ref{test19f}).
 Similarly, Theorem~\ref{thm:2b} with $a=3$ and $b=1$ yields  there is no suitable multipler $M\in\mathcal{M}_{odd}$ for $1/k+G$ with $k\geq 39.93$. Specifically the phase $(3\angle (1/k+G(j\omega))-\angle(1/k+G(3j\omega)))/3$ exceeds $180^o$ when $k\geq 32.61$.

These results show a non-trivial improvement over those in \cite{Wang:18}. While it should be possible to achieve identical results using either the condition of \cite{Jonsson:96} or that  of  \cite{Wang:18} (see Appendix), the conditions of Theorems~\ref{thm:2a} and~\ref{thm:2b} can be applied in a systematic manner. Fig~\ref{test19d_fig1} shows the bounds for several other values of $\zeta$ while Fig~\ref{test19d_fig2} shows the value of $a$ yielding the lowest bound for each test (the value of $b$ is $1$ 	for each case).

\begin{figure}[htbp]
\begin{center}
\includegraphics[width = 0.9\linewidth]{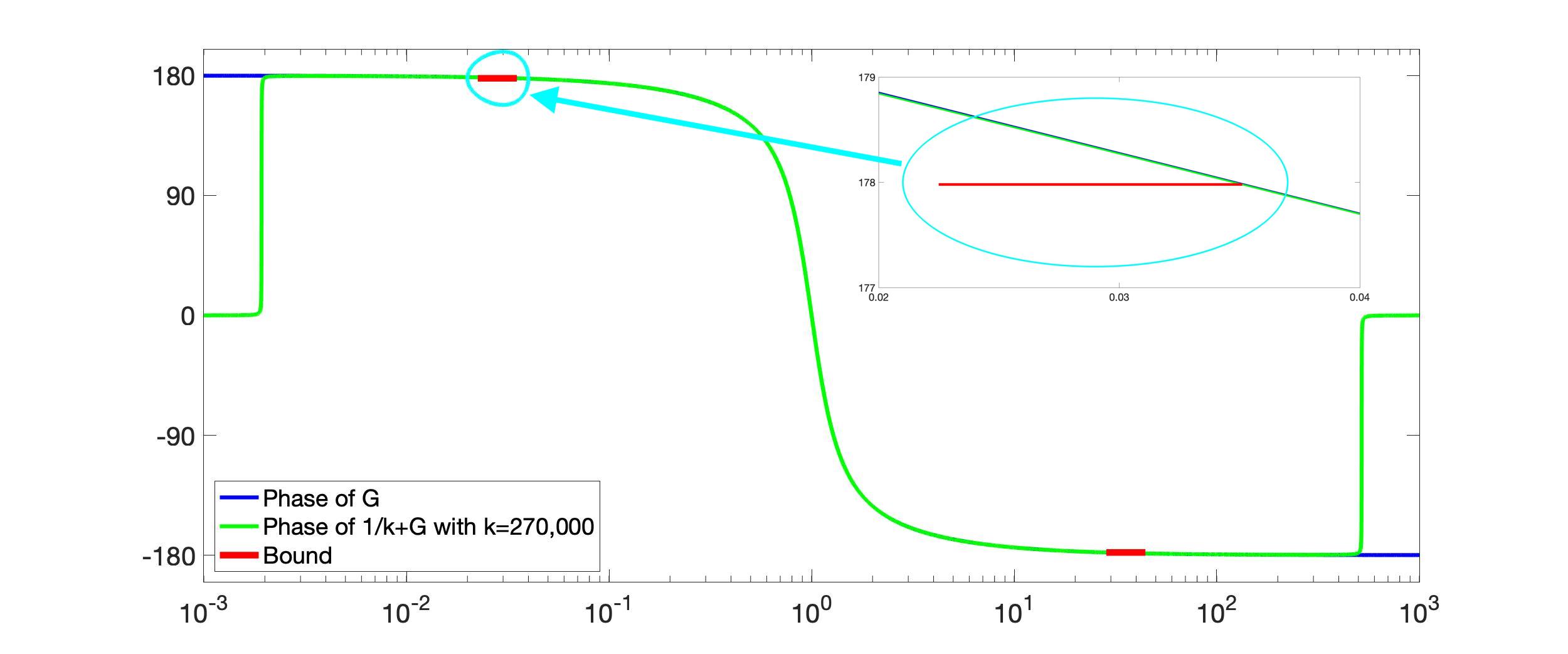}
\caption{Example 2. O'Shea's example with $\zeta=0.25$. Application of the condition in \cite{Wang:18} yields there to be no suitable multiplier $M\in\mathcal{M}$ when $k\geq 270,000$.}\label{test19e_fig1}
\end{center}
\end{figure}

\begin{figure}[htbp]
\begin{center}
\includegraphics[width = 0.9\linewidth]{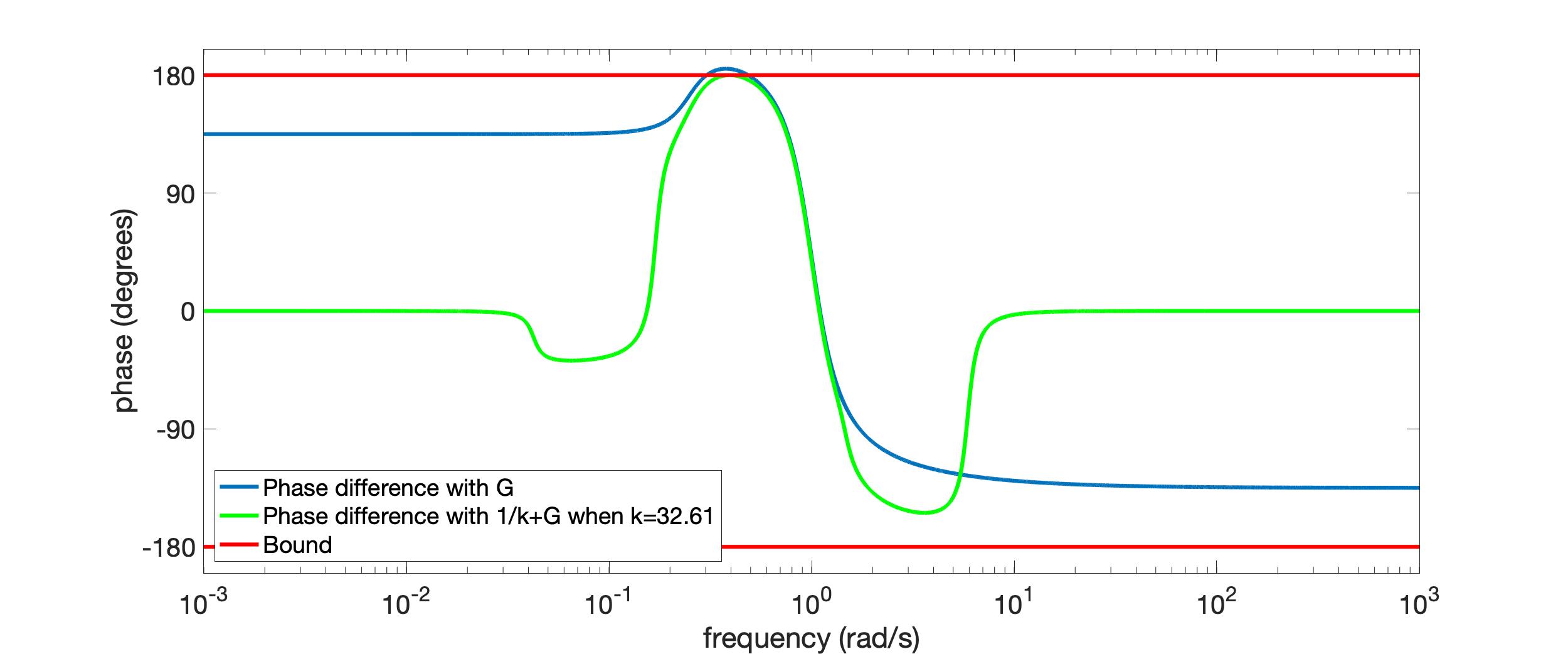}
\caption{Example 2. O'Shea's example with $\zeta=0.25$. Application of Theorem~\ref{thm:2a} with $a=4$ and $b=1$ yields there to be no suitable multiplier $M\in\mathcal{M}$ when $k\geq 32.61$.}\label{test19e_fig3}
\end{center}
\end{figure}

\begin{figure}[htbp]
\begin{center}
\includegraphics[width = 0.9\linewidth]{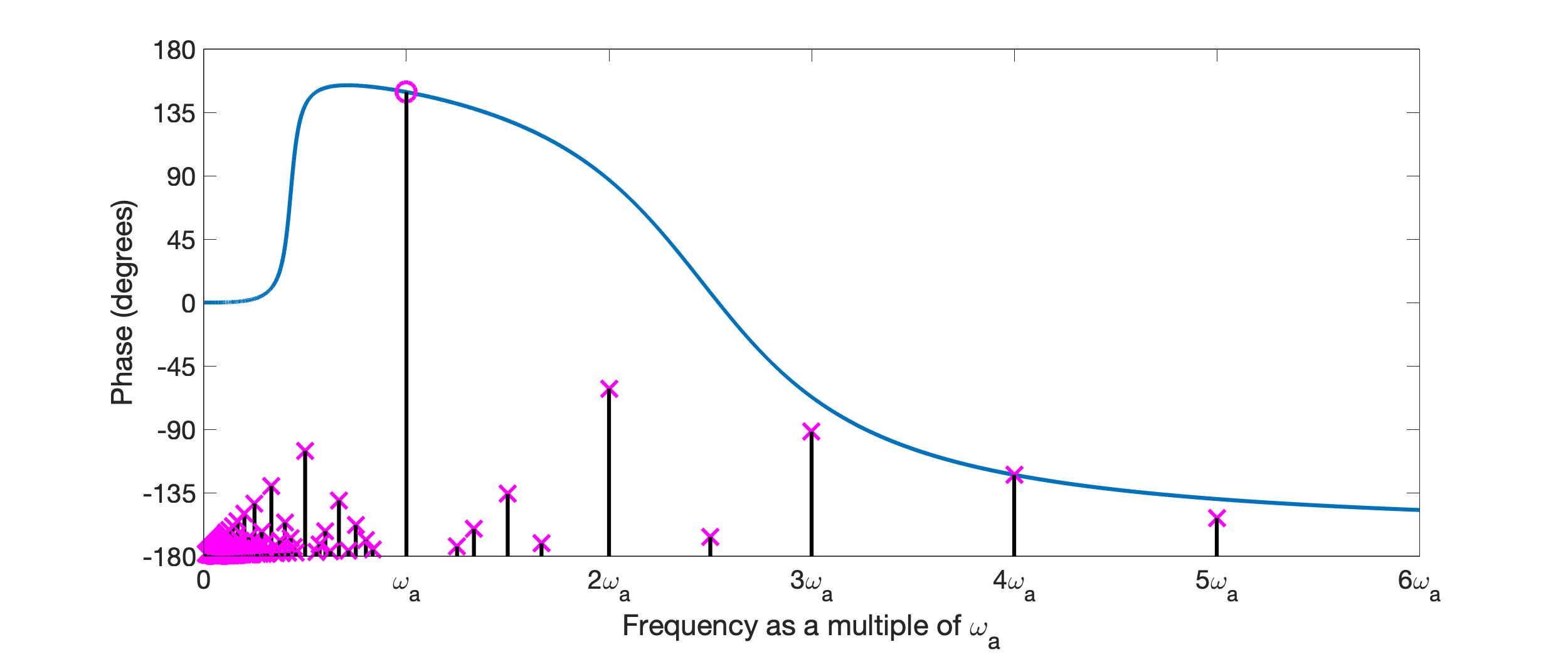}
\caption{Example 2.O'Shea's example with $\zeta=0.25$. The phase of $1/k+G(j\omega)$ with $k=32.61$ is shown. The phase of $1/k+G(j\omega_a)$ is $149.42^o$ at $\omega_a = 0.3938$ and the corresponding forbidden regions are shown (compare Fig~\ref{test02_fig1}). The phase touches the bound at $4\omega_a$.}\label{test19f}
\end{center}
\end{figure}

\begin{figure}[htbp]
\begin{center}
\includegraphics[width = 0.9\linewidth]{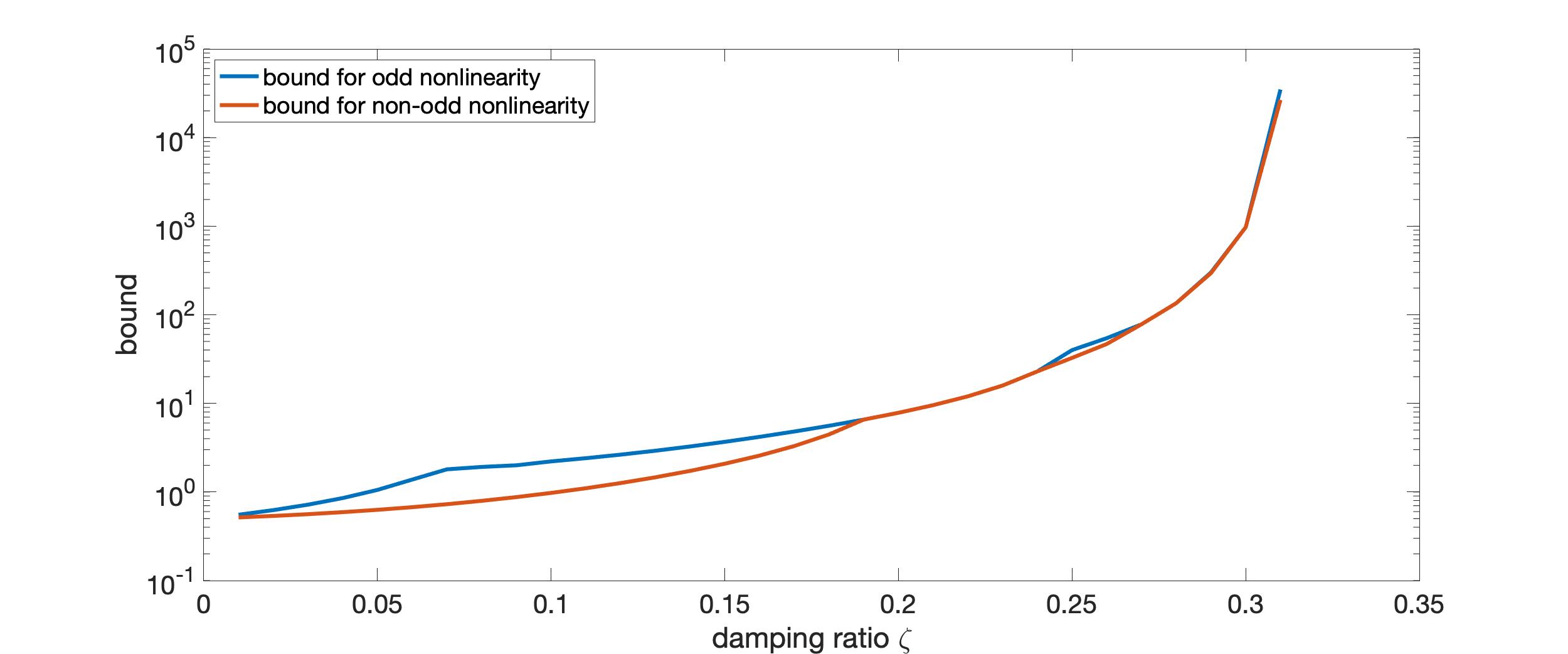}
\caption{Example 2. Bounds on the slope above which Theorem~\ref{thm:2a} or~\ref{thm:2b} guarantee there can be no suitable multiplier as damping ratio~$\zeta$ varies.}\label{test19d_fig1}
\end{center}
\end{figure}

\begin{figure}[htbp]
\begin{center}
\includegraphics[width = 0.9\linewidth]{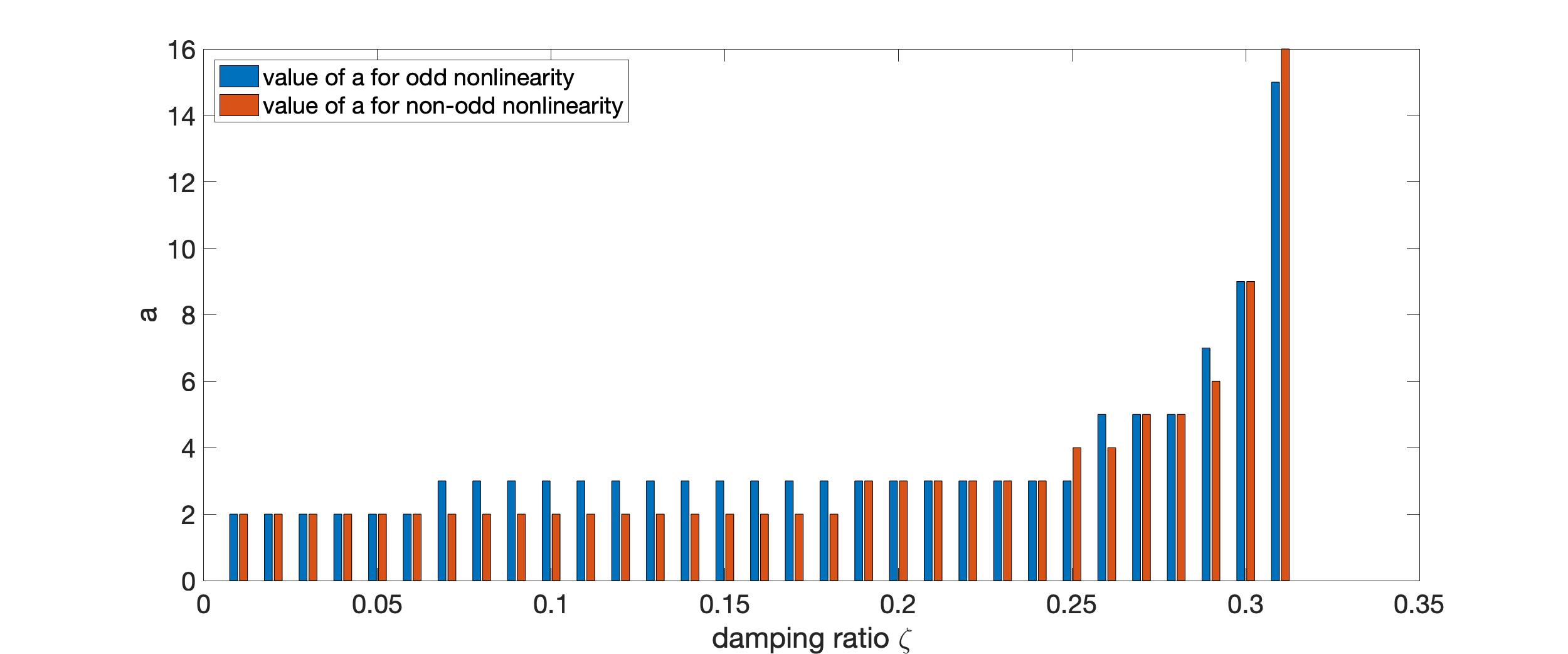}
\caption{Example 2. Values of $a$ used to find the slope bounds shown in Fig~\ref{test19d_fig1}. The value of $b$ is $1$ for all shown results.}\label{test19d_fig2}
\end{center}
\end{figure}


\subsection{Example 3}

In  \cite{Wang:18} we argue that  phase limitatons are closely linked to the Kalman Conjecture. This plays an important role in the theory of absolute stability for Lurye systems. Barabanov \cite{Barabanov88} shows it to be true for third-order systems  via a subclass of the OZF multipliers but fourth-order counterexamples are known \cite{Fitts66,Leonov13}. It is trivial that  negative imaginary systems satisfy the Kalman Conjecture  \cite{Carrasco17}. In \cite{Zhang18} we indicate via the tailored construction of OZF multipliers that second-order systems with delay satisfy the Kalman Conjecture. Until now it has remained an open question whether third-order systems with delay satisfy the Kalman Conjecture.

Consider the third-order system with delay that has transfer function
\begin{equation}
G(s) = e^{-s} \frac{s^2+0.8s+1.5}{s^3+1.2s^2+1.12s+0.32}.
\end{equation}
The Nyquist gain is $k_N=2.0931$. That is to say for all $0\leq k < k_N$ the sensitivity function 
${
\left [
	\begin{array}{cc}
		1 & G\\
		-k & 1
	\end{array}
\right ]^{-1}
}$
is stable.  Fig.~\ref{test22_fig_over} shows  $(2\angle \left (1/2+G(j\omega)\right ) - \angle \left (1/2 + G(2j\omega)\right ))/2$ against frequency. The value drops significantly below $-180^o$, and hence by Theorem~\ref{thm:2a} there is no suitable $M\in\mathcal{M}$ for $1/2+G$. 
The phases of $1/2+G(j\omega)$ and of $1/2+G(2j\omega)$ are superimposed. 
Fig.~\ref{test22_fig3} shows a time response of a Lurye system with gain $2$, a step input at time $t=0$ and simple saturation. The response appears to be periodic. The stable linear response (i.e. without saturation) is superimposed. These results indicate that this is a (first) example of a third order plant with delay which does not satisfiy the Kalman Conjecture.



\begin{figure}[htbp]
\begin{center}
\includegraphics[width = 0.9\linewidth]{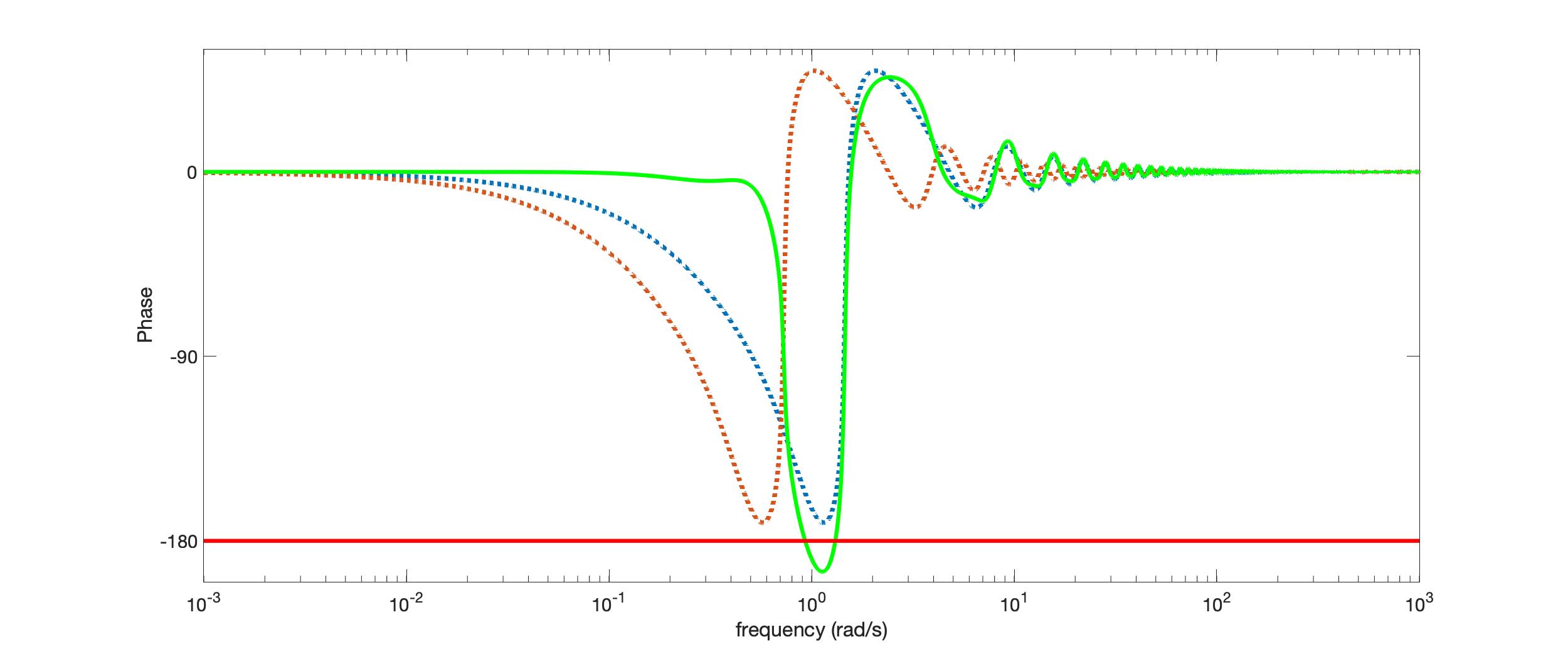}
\caption{Example 3. The value of $(2\angle \left (1/2+G(j\omega)\right )-\angle \left (1/2+G(2j\omega)\right ))/2$ drops below significantly $-180^o$ so by Theorem~\ref{thm:2a} there is no suitable multiplier. The phase of $1/2+G(j\omega)$ (blue dotted) and the phase of $1/2+G(2j\omega)$ (red dotted) are also shown.}\label{test22_fig_over}
\end{center}
\end{figure}

\begin{figure}[htbp]
\begin{center}
\includegraphics[width = 0.9\linewidth]{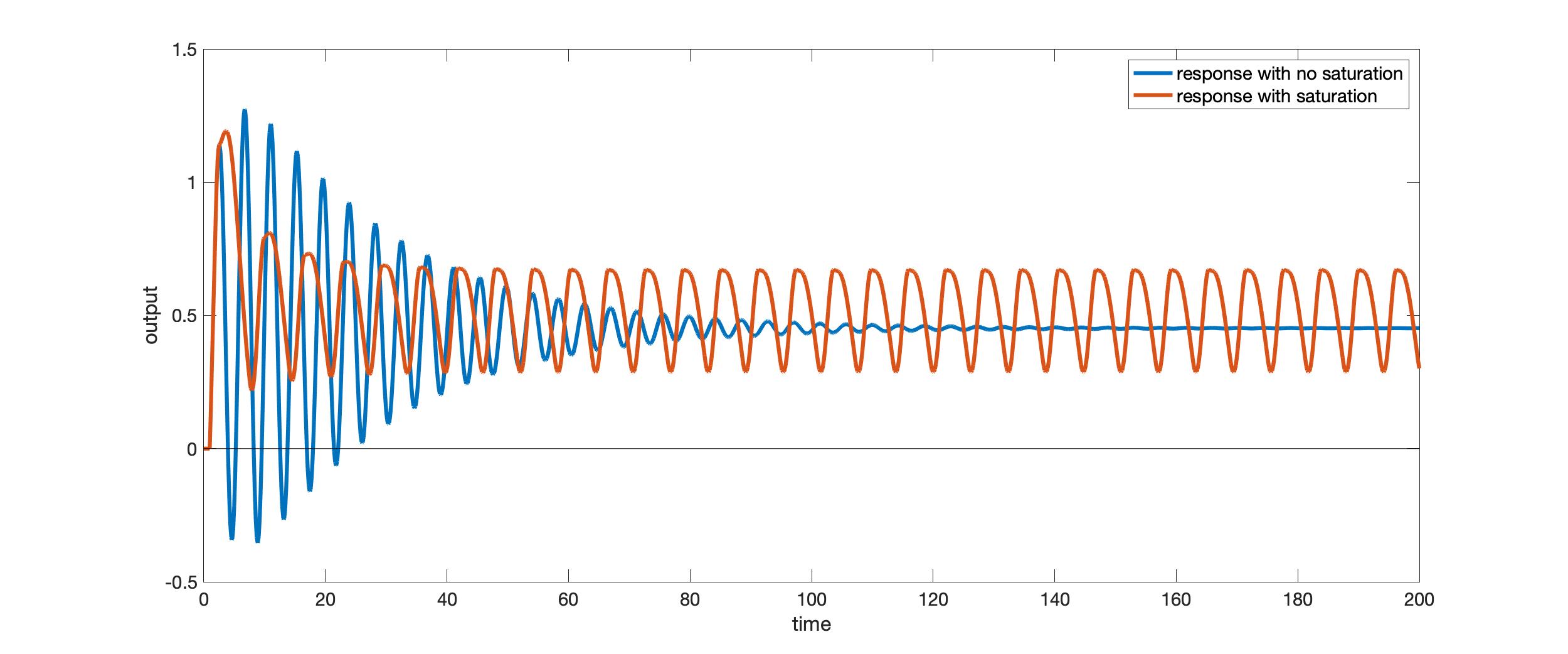}
\caption{Example 3. Time response of the Lurye system, with and without saturation.}\label{test22_fig3}
\end{center}
\end{figure}

%% file: app_proofs.tex
\appendix[Proofs]

\subsection{Proofs of Theorems~\ref{Jthm_a} and~\ref{Jthm_b}}

\begin{proof}[Proof of Theorem~\ref{Jthm_a}]
Let $M\in\mathcal{M}$ take the form of Definition~\ref{def2a}. 
Then
\begin{equation}
\begin{split}
	M(j\omega) & = m_0-\int_{-\infty}^{\infty}h(t)e^{-j\omega t}\,dt-\sum_{i=1}^{\infty}h_ie^{-j\omega t_i},\\
			& = \bar{m}_0  -\int_{-\infty}^{\infty}h(t)e^{-j\omega t}\,dt+\sum_{i=1}^{\infty}h_iM^-_{t_i}(j\omega),
\end{split}
\end{equation}
where
\begin{equation}\label{barm_ineq}
\bar{m}_0 = m_0-\sum_{i=1}^{\infty}h_i\geq \| h\|_1,
\end{equation}
and
\begin{equation}
\begin{split}
 \sum_{r=1}^N    \lambda_r  \mbox{Re} & \left \{M(j\omega_r) 
				 G(j\omega_r)
			\right \} 
=
\bar{m}_0 \sum_{r=1}^N   \lambda_r  \mbox{Re}\left \{
				 G(j\omega_r)
			\right \} \\
& -  \int_{-\infty}^{\infty}h(t)
 \sum_{r=1}^N   \lambda_r  \mbox{Re}\left \{e^{-j\omega_r t}
				 G(j\omega_r)
			\right \} \, dt\\
&
+ \sum_{i=1}^{\infty}h_i
 \sum_{r=1}^N   \lambda_r  \mbox{Re}\left \{M^-_{t_i}(j\omega_r)
				 G(j\omega_r)
			\right \}.
\end{split}
\end{equation}
Suppose the conditions of Theorem~\ref{Jthm_a} hold. Then, by (\ref{thm1_ineq}),
\begin{equation}
\begin{split}
 \sum_{r=1}^N    \lambda_r  \mbox{Re} & \left \{M(j\omega_r) 
				 G(j\omega_r)
			\right \} 
\leq
\bar{m}_0 \sum_{r=1}^N   \lambda_r  \mbox{Re}\left \{
				 G(j\omega_r)
			\right \} \\
& -  \int_{-\infty}^{\infty}h(t)
 \sum_{r=1}^N   \lambda_r  \mbox{Re}\left \{e^{-j\omega_r t}
				 G(j\omega_r)
			\right \} \, dt\label{no_sum}
\end{split}
\end{equation}
In addition, we can write (\ref{thm1_ineq}) as
\begin{align}\label{thm1_ineq_alt}
\sum_{r=1}^N\lambda_r & \mbox{Re}\left \{
				 G(j\omega_r)
			\right \}  \nonumber\\
& \leq 
\sum_{r=1}^N\lambda_r \mbox{Re}\left \{e^{-j\omega_r\tau}
				 G(j\omega_r)
			\right \} \mbox{ for all } \tau\in \mathbb{R}\backslash 0.
\end{align}
Averaging this expression over $\tau$ yields
\begin{equation}\label{ineq2}
	\begin{split}
\sum_{r=1}^N\lambda_r & \mbox{Re}\left \{
				 G(j\omega_r)
			\right \}
	 =\lim_{T\rightarrow\infty}\frac{1}{T}\int_0^T \sum_{r=1}^N\lambda_r \mbox{Re}\left \{
				 G(j\omega_r)
			\right \}\, dt\\
	& \leq \lim_{T\rightarrow\infty}\frac{1}{T}\int_0^T\sum_{r=1}^N\lambda_r \mbox{Re}\left \{e^{-j\omega_r\tau}
				 G(j\omega_r)
			\right \}\,dt \\
	& = \sum_{r=1}^N\lambda_r \mbox{Re}\left \{ \lim_{T\rightarrow\infty}\frac{1}{T}\int_0^Te^{-j\omega_r\tau}\, dt\,
				 G(j\omega_r)
			\right \} \\
& = 0.
	\end{split}
\end{equation}
From  (\ref{barm_ineq}) and (\ref{ineq2})  we obtain
\begin{equation}
 \bar{m}_0
\sum_{r=1}^N \lambda_r \mbox{Re}\left \{
				 G(j\omega_r)\right \}
\leq 
\|h\|_1
\sum_{r=1}^N \lambda_r \mbox{Re}\left \{
				 G(j\omega_r)
			\right \}.
\end{equation}
This, with (\ref{thm1_ineq_alt}), yields
\begin{align}\label{no_sum2}
 \bar{m}_0
\sum_{r=1}^N \lambda_r & \mbox{Re}\left \{
				 G(j\omega_r)\right \}\nonumber\\
& \leq
 \int_{-\infty}^{\infty} h(t)
\sum_{r=1}^N \lambda_r \mbox{Re}\left \{e^{-j\omega_r t}
				 G(j\omega_r)
			\right \}\,dt.
\end{align}
Together (\ref{no_sum}) and (\ref{no_sum2}) yield
\begin{equation}\label{final_thm1a}
\sum_{r=1}^N \lambda_r \mbox{Re}\left \{M(j\omega_r )
				 G(j\omega_r)
			\right \} \leq 0.
\end{equation}
It follows from Definition~\ref{def1} that $M$ is not suitable for $G$.
\end{proof}
\begin{proof}[Proof of Theorem~\ref{Jthm_b}]
Let $M\in\mathcal{M}$ take the form of Definition~\ref{def2b}. 
Define $\mathcal{H}^+=\{i\in\mathbb{Z}^+\mbox{ such that }h_i\geq 0\}$  and $\mathcal{H}^-=\{i\in\mathbb{Z}^+\mbox{ such that }h_i< 0\}$.
Then
\begin{equation}
\begin{split}
	M(j\omega)   =  & \bar{m}_0  -\int_{-\infty}^{\infty}h(t)e^{-j\omega t}\,dt\\
& +\sum_{i\in\mathcal{H}^+}^{\infty}h_iM^-_{t_i}(j\omega)+\sum_{i\in\mathcal{H}^-}^{\infty}|h_i|M^+_{t_i}(j\omega)
\end{split}
\end{equation}
where this time
\begin{equation}\label{barm_ineq_b}
\bar{m}_0 = m_0-\sum_{i=1}^{\infty}|h_i|\geq \| h\|_1.
\end{equation}
Suppose the conditions of both Theorem~\ref{Jthm_a} and~\ref{Jthm_b} hold. Then (\ref{thm1_ineq}) and (\ref{thm1b_ineq})) yield (\ref{no_sum}) as before, but with $\bar{m}_0$ given by (\ref{barm_ineq_b}). Furthermore,
we can write (\ref{thm1_ineq}) and (\ref{thm1b_ineq}) together as
\begin{equation}\label{thm1_ineq_alt_b}
\begin{split}
\sum_{r=1}^N & \lambda_r \mbox{Re}\left \{
				 G(j\omega_r)
			\right \} \\
& \leq 
-
\sum_{r=1}^N\lambda_r \left | \mbox{Re}\left \{e^{-j\omega_r\tau}
				 G(j\omega_r)
			\right \} 
\right |\mbox{ for all } \tau\in \mathbb{R}\backslash 0.
\end{split}
\end{equation}
Since (\ref{ineq2}) still holds, from
 (\ref{barm_ineq_b}), (\ref{thm1_ineq_alt_b}) and (\ref{ineq2})  we obtain
\begin{align}\label{no_sum2b}
 \bar{m}_0
\sum_{r=1}^N  & \lambda_r \mbox{Re}\left \{
				 G(j\omega_r)\right \}\nonumber\\
&
+\int_{-\infty}^{\infty} |h(t)|
\sum_{r=1}^N \lambda_r \left | \mbox{Re}\left \{e^{-j\omega_r t}
				 G(j\omega_r)
			\right \}
\right |
\,dt \leq 0.
\end{align}
Together (\ref{no_sum}) and (\ref{no_sum2b}) yield  (\ref{final_thm1a}) as before.
It follows from Defintion~\ref{def1} that $M$ is not suitable for $G$.
\end{proof}



\subsection{Proof of Theorems~\ref{thm:2a} and \ref{thm:2b}}

\input{Part5_Jonsson01}


\subsection{Proofs of Corollaries~\ref{m_corollary_a} and ~\ref{m_corollary_b}}

\begin{proof}[Proof of Corollary~\ref{m_corollary_a}]
Without loss of generality let $a<b$.
The result follows by setting the intervals
\begin{align}
	[\alpha,\beta] =[a\omega_0 - \varepsilon, a\omega_0 + \varepsilon]\mbox{ and }[\gamma,\delta] =[b\omega_0 - \varepsilon, b\omega_0 + \varepsilon]
\end{align}
with $\varepsilon>0$ and taking the limit as $\varepsilon \rightarrow 0$. Specifically we find
\begin{equation}
\begin{split}
\psi(t) 
& = \frac{2\lambda}{t}\sin(a\omega_0t)\sin(\varepsilon t)
-\frac{2\mu}{t}\sin(b\omega_0t)\sin(\varepsilon t)
\\
\phi(t) & = 2\varepsilon\lambda+2\varepsilon\kappa\mu+\phi_1(t)\\
\phi_1(t) 
& = -\frac{2\lambda}{t}\cos(a\omega_0t)\sin(\varepsilon t)
-\frac{2\kappa\mu}{t}\cos(b\omega_0t)\sin(\varepsilon t),
\end{split}
\end{equation}
with $a\lambda=b\mu$.
Hence
\begin{equation}
 \overline{\rho}^c  = \lim_{\varepsilon\rightarrow 0}\rho^c\\
\end{equation}
\end{proof}

\begin{proof}[Proof of Corollary~\ref{m_corollary_b}]
In addition
\begin{equation}
\tilde{\phi}(t)  = 2\varepsilon\lambda+2\varepsilon\kappa\mu-|\phi_1(t)|
\end{equation}
and hence
\begin{equation}
\overline{\rho}^c_{\mbox{odd}}  = \lim_{\varepsilon\rightarrow 0}\rho^c_{\mbox{odd}}
\end{equation}
\end{proof}


\subsection{Proof of Theorems~\ref{Meg_equiv_a} and~\ref{Meg_equiv_b} }

\input{Part6_Megretski01}

%% file: Part5_Jonsson01.tex


In the following we apply Theorems~\ref{Jthm_a} and~\ref{Jthm_b} with $N=2$. Furthermore, we assume $\omega_2/\omega_1$ is rational, i.e. that there is some $\omega_0>0$ and integers $a$ and $b$ such that either $\omega_1=a \omega_0$ and $\omega_2=b\omega_0$ or $\omega_1=b \omega_0$ and $\omega_2=a\omega_0$. 
We begin with two technical lemmas.
\begin{subtheorem}{lemma}
\begin{lemma}\label{lem1a} Let $a$ and $b$ be coprime positive integers and
\begin{equation}\label{def_f1}
\begin{split}
	f_1(\omega) & = -b \sin\theta (\cos\phi-\cos(\phi-aw))\\
	f_2(\omega) & = -a\sin\phi(\cos\theta-\cos(\theta+b\omega))
\end{split}
\end{equation}
with $\omega\in\mathbb{R}$ and $\theta,\phi\geq 0$. Then
\begin{equation}f_1(\omega)+f_2(\omega) \leq 0 \mbox{ for all }\omega,\end{equation}
provided
\begin{equation}\label{new_p_pi}
a\theta+b\phi < p \pi,
\end{equation}
with $p=1$.
\end{lemma}
\begin{lemma}\label{lem1b} Let $a$ and $b$ be coprime positive integers and
\begin{equation}\label{def_f3}
\begin{split}
	f_3(\omega) & = -b \sin\theta (\cos\phi+\cos(\phi-aw))\\
	f_4(\omega) & = -a\sin\phi(\cos\theta+\cos(\theta+b\omega))
\end{split}
\end{equation}
with $\omega\in\mathbb{R}$ and $\theta,\phi\geq 0$. Then
\begin{equation}f_3(\omega)+f_4(\omega) \leq 0 \mbox{ for all }\omega,\end{equation}
provided (\ref{new_p_pi}) holds with $p=1$ when $a$ and $b$ are both odd and $p=1/2$ when either $a$ or $b$ is even.
\end{lemma}
\end{subtheorem}

\begin{proof}[Proof of Lemma \ref{lem1a}]
The term $f_1(\omega)$ is only positive when $[a\omega]_{2\pi}\in(0,2\phi)$. Similarly the term $f_2(\omega)$ is only positive when $[-b\omega]_{2\pi}\in(0,2\theta)$. When $p=1$ there is no $\omega$ such that $f_1(\omega)$ and $f_2(\omega)$ are simultaneously positive. Specifically, suppose $\omega$ is a frequency such that  $a\omega+2m\pi\in(0,2\phi)$ and $-b\omega+2n\pi\in(0,2\theta)$ for some integers $m$ and $n$.  Then $2(mb+na)\pi\in(0,2p\pi)$. This cannot be the case with $p<1$; when $a$ and $b$ are coprime then it can be satisfied with $p>1$ provided $m$ and $n$ are chosen such that $mb+n=1$.

Hence, with $p=1$, it suffices to show that $f_1(\omega)+f_2(\omega)\leq 0$ when $f_1(\omega)\geq 0$, i.e. on the intervals  $0\leq[a\omega]_{2\pi}\leq2\phi$. A similar argument will follow by symmetry for intervals where $f_2(\omega)\geq 0$.
\begin{figure}[htbp]
\begin{center}
\includegraphics[width = 0.9\linewidth]{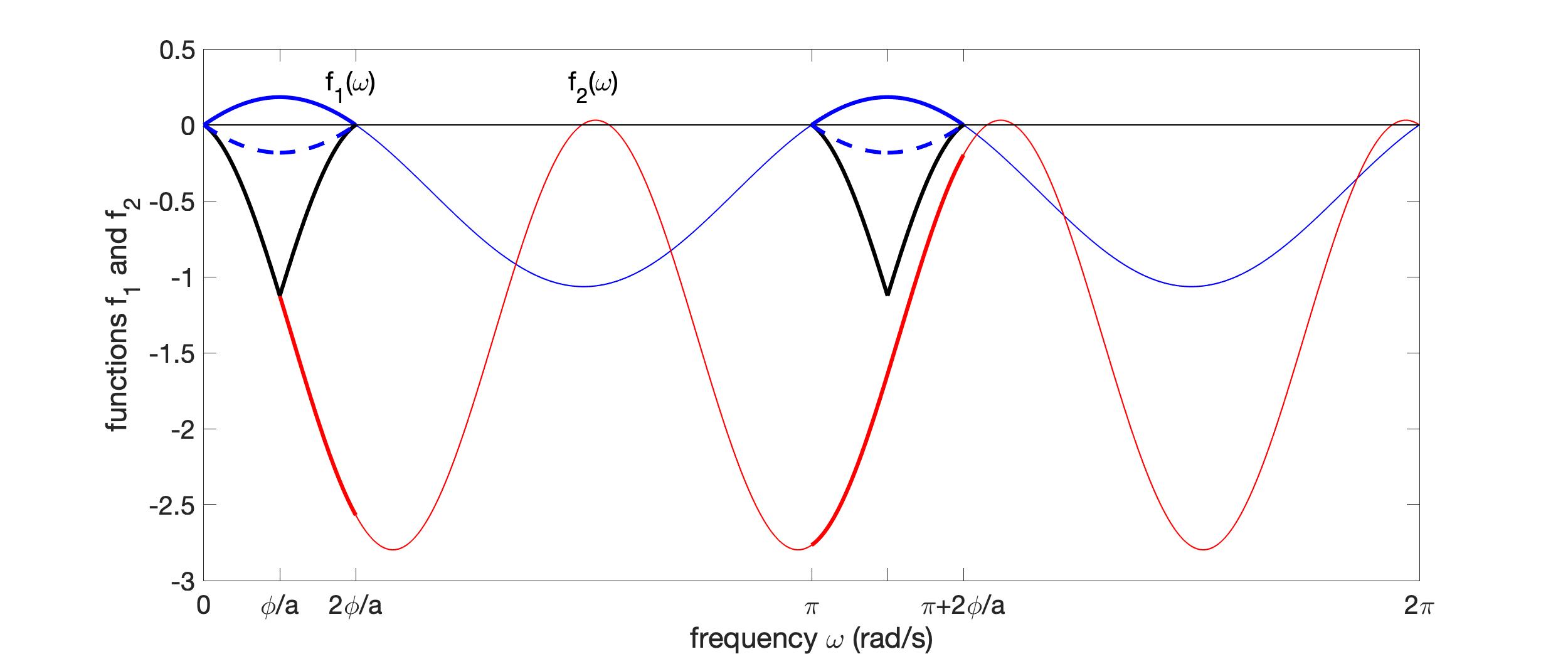}
\caption{Illustration of Lemma~\ref{lem1a} with $a=2$, $b=3$, $\theta=\pi/15$ and $\phi =\pi/4$. The functions $f_1(\cdot)$ and $f_2(\cdot)$ are never simultaneously positive. We have the relations  $f_1(\omega)=f_1(2\phi/a-\omega)$ when $\phi/a\leq\omega\leq2\phi/a$ and also $f_1(\omega)=f_1(\omega-\pi)$ when $\pi\leq\omega\leq\pi+2\phi/a$.
Similarly $f_2(\omega)\leq f_2(2\phi/a-\omega)$ when $\phi/a\leq\omega\leq2\phi/a$,  $f_2(\omega)\leq f_2(\omega-\pi)$ when $\pi\leq\omega\leq\pi+\phi/a$ and  $f_2(\omega)\leq f_2(\pi+2\phi/a-\omega)$ when $\pi +\phi/a \leq\omega\leq\pi+2\phi/a$. Hence to show $f_1(\omega)+f_2(\omega)\leq 0$ when $f_1(\omega)\geq 0$, it suffices to consider the interval $0\leq\omega\leq \phi/a$.}
\end{center}
\end{figure}

Consider first the interval $a\omega\in[0,\phi]$.
We have
\begin{equation}\begin{split}
\frac{df_1}{d\omega}(\omega)  & =   ab\sin\theta \sin(\phi-a\omega)\\
\frac{df_2}{d\omega}(\omega)  & =-ab\sin\phi\sin(\theta+b\omega)
\end{split}\end{equation}
But
\begin{equation}\begin{split}
	\sin (\phi-a\omega) & \leq \sin\phi -a\omega \cos \phi\mbox{ (by slope restriction), and}\\
    \sin(\theta+b\omega) & \geq  \sin\theta +\frac{a\omega}{\phi}\left [\sin\left ( \theta+\frac{b\phi}{a}\right )-\sin\theta\right ]\\
&  \mbox{\hspace{3 cm} (by local convexity)}.
\end{split}\end{equation}
Hence
\begin{equation}\begin{split}
\frac{df_1}{d\omega}(\omega)+\frac{df_2}{d\omega}(\omega) 
  \leq  & -a^2b\omega\sin\theta  \cos \phi\\
&   -\frac{a^2b\omega}{\phi}\left [\sin\left ( \theta+\frac{b\phi}{a}\right )-\sin\theta\right ]\sin\phi\\
	 \leq & 0.
\end{split}\end{equation}
Since $f_1(0)=f_2(0)=0$ if follows that $f_1(\omega)+f_2(\omega)\leq 0$ on the interval $a\omega\in[0, \phi]$.
\begin{figure}[htbp]
\begin{center}
\includegraphics[width = 0.9\linewidth]{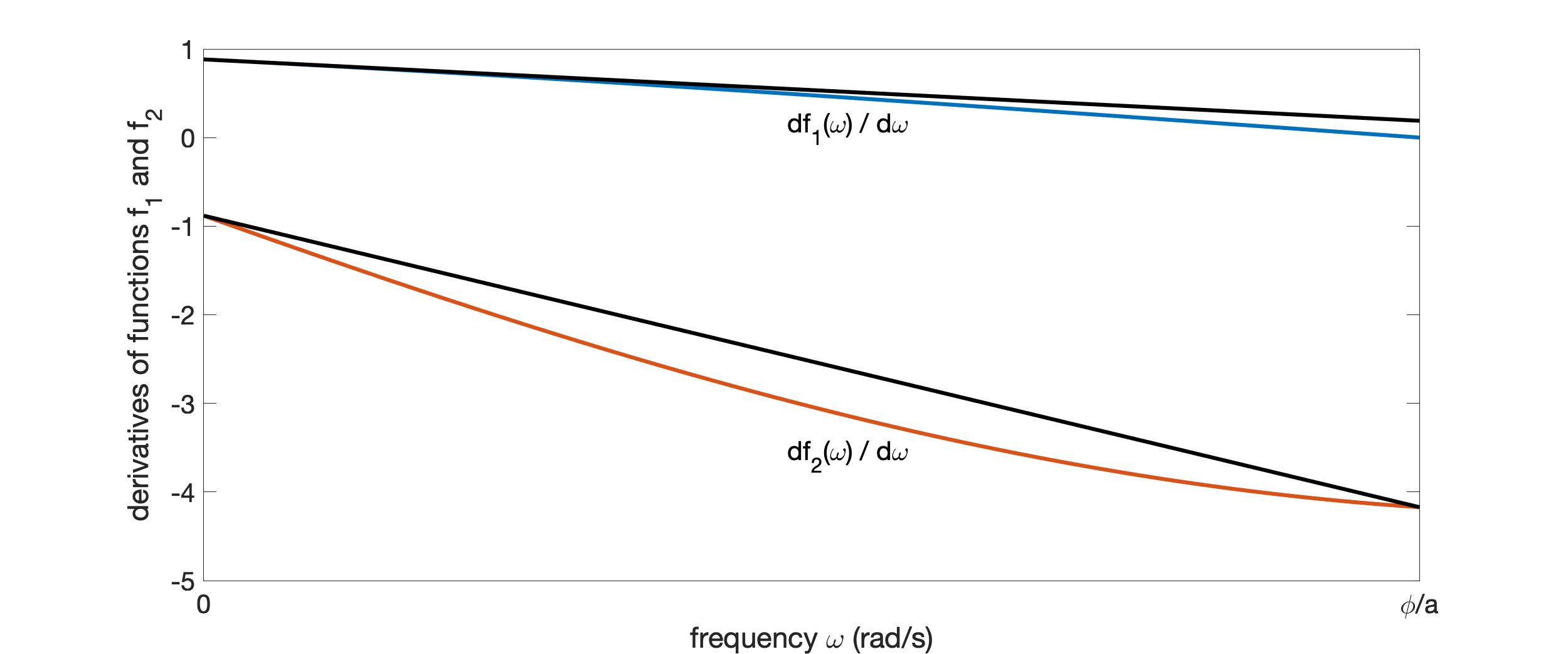}
\caption{Illustration of Lemma~\ref{lem1a} with $a=2$, $b=3$, $\theta=\pi/15$ and $\phi =\pi/4$. On the interval $0\leq\omega\leq \phi/a$ the derivative of $f_1(\cdot)$ is bounded above by its gradient at $\omega=0$ while the derivative of $f_2(\cdot)$ is bounded above by the chord joining its two end points. It follows that $f_1(\cdot)+f_2(\cdot)$ is non-positive on this interval. }
\end{center}
\end{figure}

Consider next the interval $a\omega\in[\phi,2 \phi]$. By symmetry $f_1(\omega) = f_1(2\phi-\omega)$ on this interval. Since $f_2(\omega)\leq 0$ on this interval we must have $f_2(\omega) \leq f_2(2\phi-\omega)$ on this same interval. Hence $f_1(\omega)+f_2(\omega)\leq 0$ on the interval $a\omega\in[\phi, 2\phi]$.

Similar arguments follow: firstly on the intervals $[a\omega]_{2\pi}\in[0,\phi]$ where $f_1(\omega) = f_1([a\omega]_{2\pi}/a)$ and $f_2(\omega) \leq f_2([a\omega]_{2\pi}/a)$; secondly  on the intervals $[a\omega]_{2\pi}\in[\phi,2\phi]$ where $f_1(\omega) = f_1(2\phi-[a\omega]_{2\pi}/a)$ and $f_2(\omega) \leq f_2(2\phi-[a\omega]_{2\pi}/a)$.
\end{proof}
\begin{proof}[Proof of Lemma~\ref{lem1b}]
The term $f_3(\omega)$ is only positive when $[a\omega]_{2\pi}\in(\pi,\pi+2\phi)$. Similarly the term $f_4(\omega)$ is only positive when $[-b\omega]_{2\pi}\in(\pi,\pi+2\theta)$. Let us consider conditions for which they are simultaneously positive. Suppose $\omega$ is a frequency such that  $a\omega+2m\pi\in(\pi,\pi+2\phi)$ and $-b\omega+2n\pi\in(\pi,\pi+2\theta)$ for some integers $m$ and $n$.  Then $2(mb+na)\pi\in((a+b)\pi,(a+b+2p)\pi)$. If $a$ and $b$ are both odd, then $a+b$ is even and hence this can only be true when $p>1$. By contrast, if either $a$ or $b$ is even (but not both, as they are coprime) then $a+b$ is odd and we can choose $mb+na=a+b+1$ when $p>1/2$.

 It then follows that $f_3(\omega)+f_4(\omega)\leq 0$ for all $\omega$ by an argument similar to that in the proof of Lemma~\ref{lem1a}.
\begin{figure}[htbp]
\begin{center}
\includegraphics[width = 0.9\linewidth]{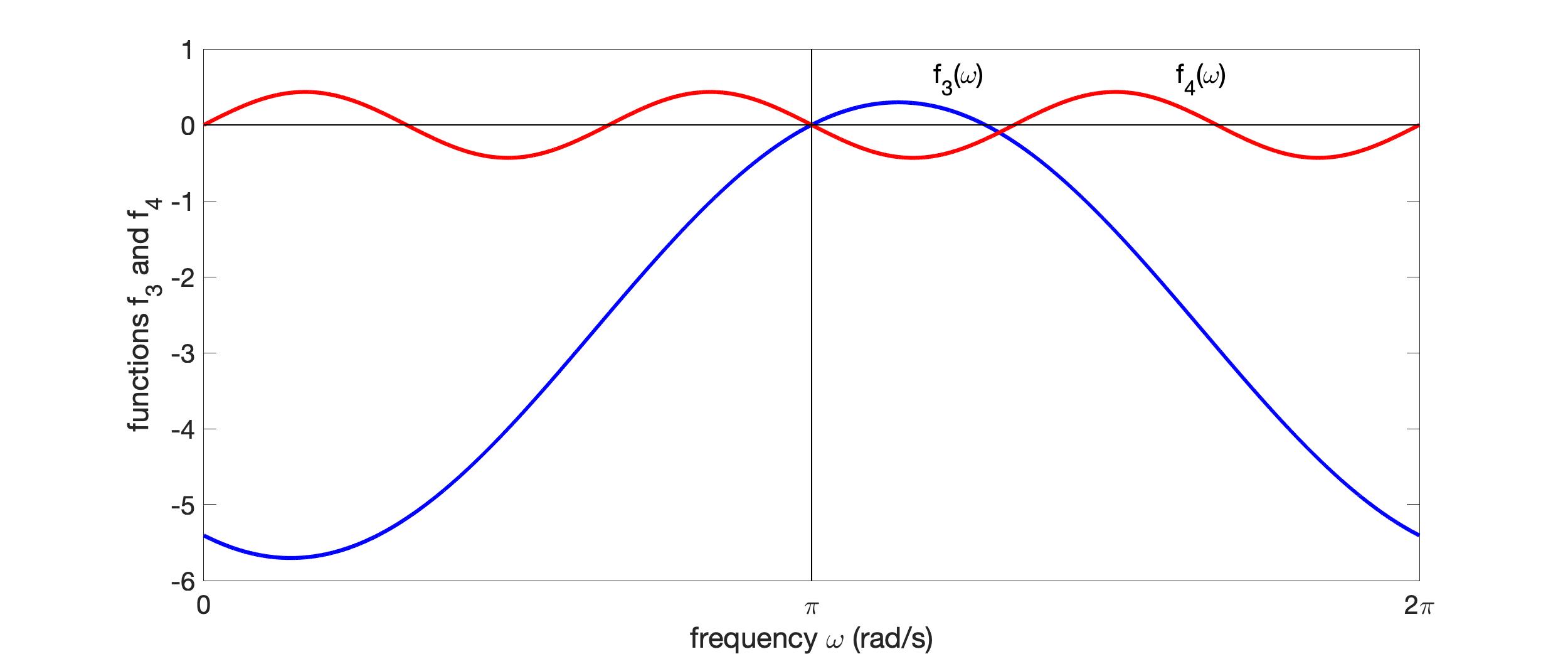}
\caption{Illustration of Lemma~\ref{lem1a} with $a=1$, $b=3$, $\theta=\pi/2$ and $\phi =\pi/7$. The functions $f_3(\cdot)$ and $f_4(\cdot)$ are never simultaneously positive. The function $f_3(\omega)$ is non-negative on the interval $\pi\leq\omega\leq\pi+2\phi/a$. The function $f_4(\omega)$ is non-negative on the interval $\pi-2\theta/b\leq \omega\leq\pi$.}
\end{center}
\end{figure}

\begin{figure}[htbp]
\begin{center}
\includegraphics[width = 0.9\linewidth]{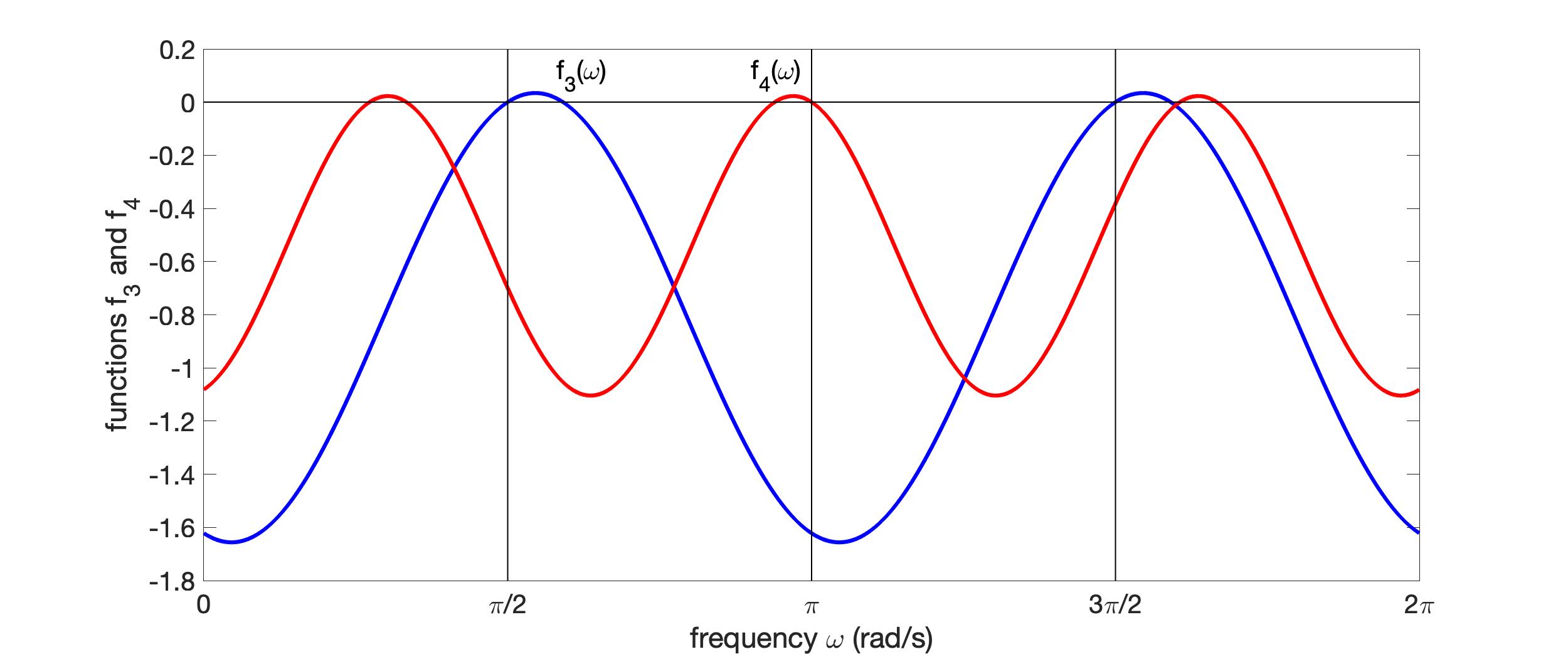}
\caption{Illustration of Lemma~\ref{lem1a} with $a=2$, $b=3$, $\theta=\pi/11$ and $\phi =\pi/11$. The functions $f_3(\cdot)$ and $f_4(\cdot)$ are never simultaneously positive. The function $f_3(\omega)$ is non-negative on the interval $\pi/2\leq\omega\leq\pi/2+2\phi/a$. The function $f_4(\omega)$ is non-negative on the interval $\pi-2\theta/b\leq \omega\leq\pi$.}
\end{center}
\end{figure}

\end{proof}

\begin{proof}[Proof of Theorem \ref{thm:2a}]
Without loss of generality suppose $a$ and $b$ are coprime, and consider the case where 
$
b\angle G(a j\omega_0 ) > a \angle G(b j\omega_0)$.
Put
\begin{align}\label{G_def}
G(a j\omega_0 ) & = g_a e^{j(\pi-\phi)}\mbox{ and }\nonumber\\
G(b j\omega_0 ) & = g_b e^{j(-\pi+\theta)}\mbox{ with }
\theta,  \phi,  g_a,g_b\in\mathbb{R}^+,
\end{align}
and 
\begin{equation}\label{p_ineq}
a\theta+b\phi < p\pi,
\end{equation}
so that (\ref{G_ineq}) holds.
Immediately we have
\begin{equation}G(a j\omega_0 ) = -g_a e^{-j\phi}\mbox{ and }G(b j\omega_0 ) = -g_b e^{j\theta}.\end{equation}
Theorem~\ref{Jthm_a} then states that if there  exist non-negative $\lambda_a, \lambda_b$, with $\lambda_a+\lambda_b>0$, such that
\begin{align}\label{N=2a}
\lambda_a \mbox{Re} & \left \{
				 M^-_{\tau}(a j\omega_0) G(a j\omega_0)
			\right \}\nonumber\\
& +
\lambda_b \mbox{Re}\left \{
				 M^-_{\tau}(b j\omega_0 ) G(b j\omega_0)
			\right \} \leq 0 \mbox{ for all }M^-_{\tau}\in\mathcal{M}^-,
\end{align}
then there is no suitable  $M\in\mathcal{M}$ for $G$.

If we set $\omega=\tau \omega_0$ we can write this $f(\omega)\leq 0$ for all $\omega$ with
\begin{align}\label{f_def1}
			f(\omega) =  -\lambda_a g_a &  (1-\cos a \omega)\cos\phi+\lambda_a g_a \sin a \omega \sin \phi\nonumber\\
			& -\lambda_b g_b (1-\cos b \omega)\cos\theta-\lambda_b g_b \sin b \omega \sin \theta.
\end{align}

Choose 
\begin{equation}\label{def_lam}
\lambda_a = g_b b \sin \theta\mbox{ and }\lambda_b = g_a a \sin \phi.
\end{equation}
Then
\begin{equation}\label{def_f}
f(\omega) = g_a g_b (f_1(\omega) + f_2(\omega))
\end{equation}
with $f_1$ and $f_2$ given by (\ref{def_f1}). Hence by Lemma~\ref{lem1a}
%
$f(\omega)\leq 0$ for all $\omega$ when $p=1$.

\end{proof}
\begin{proof}[Proof of Theorem \ref{thm:2b}]
As with Theorem~\ref{thm:2a}, suppose without loss of generality that $a$ and $b$ are coprime, and consider the case where 
$
b\angle G(a j\omega_0 ) > a \angle G(b j\omega_0)$. Let $G(a j\omega_0 )$ and $G (b j\omega_0 )$ be given by (\ref{G_def})  with (\ref{p_ineq})
so that (\ref{G_ineq}) holds.
Theorem~\ref{Jthm_b} then states that if there  exist non-negative $\lambda_a, \lambda_b$, with $\lambda_a+\lambda_b>0$, such that (\ref{N=2a}) holds and in addition
\begin{align}\label{N=2b}
& \lambda_a  \mbox{Re}  \left \{
				 M^-+{\tau}(a j\omega_0) G(a j\omega_0)
			\right \}\nonumber\\
& 
+
\lambda_b \mbox{Re}\left \{
				 M^-+{\tau}(b j\omega_0) G(b j\omega_0)
			\right \} \leq 0 \mbox{ for all }M^+_{\tau}\in\mathcal{M}^+,
\end{align}
then there is no suitable  $M\in\mathcal{M}_{\mbox{odd}}$ for $G$.

For condition (\ref{N=2a}) the analysis is the same as for Theorem~\ref{thm:2a}; hence we require $p\leq 1$. We can write condition (\ref{N=2b})  as $f(\omega)\leq 0$ for all $\omega$ with
\begin{equation}\begin{split}
			f(\omega) =  -\lambda_a g_a &  (1+\cos a \omega)\cos\phi-\lambda_a g_a \sin a \omega \sin \phi\\
			& -\lambda_b g_b (1+\cos b \omega)\cos\theta+\lambda_b g_b \sin b \omega \sin \theta.
\end{split}\end{equation}
with (\ref{p_ineq}).
As before, choose $\lambda_a$  and $\lambda_b$ according to (\ref{def_lam}). 
Then
\begin{equation}\label{def_f}
f(\omega) = g_a g_b (f_3(\omega) + f_4(\omega))
\end{equation}
with $f_3$ and $f_4$ given by (\ref{def_f3}). Hence by Lemma~\ref{lem1b}
%
$f(\omega)\leq 0$ for all $\omega$ when $p=1$ if both $a$ and $b$ are odd and when $p=1/2$ if either $a$ or $b$ are even.


\end{proof}

%% file: Part6_Megretski01.tex
\begin{proof}[Proof of Theorem \ref{Meg_equiv_a}]
Consider $q_-(t)$ on $t>0$. Since $q_-(t)$ is periodic it suffices to consider the interval $0< t \leq 2\pi$. Define 
\begin{equation}
r_-(t) = b \arctan q_-(t) + a \arctan \kappa q_-(t).
\end{equation}
We will show that  for each $\kappa$ all turning points of $r_-(t)$ are bounded by $\pm (a+b-2)\frac{\pi}{2}$ and that at least one turning point touches the bounds. 
This is sufficient to establish the equivalence between Corollary~\ref{m_corollary_a} and  Corollary~\ref{cor:2a}, which is in turn equivalent to Theorem~\ref{thm:2a}.


The turning points of $r_-(t)$ occur at the same values of $t$ as the turning points of $q_-(t)$. Specifically 
\begin{equation}
\frac{d}{dt} r_-(t) = \left (\frac{b}{1+q_-(t)^2} + \frac{a\kappa}{1+\kappa^2q_-(t)^2}\right ) \frac{d}{dt}q_-(t).
\end{equation}
When $[t]_\pi\neq 0$ the derivative of $q_-(t)$ is given by
\begin{equation}
\frac{d}{dt}q_-(t)  =ab \frac{m_-(t)n_-(t)}{d_-(t)^2}
\end{equation}
with
\begin{equation}
\begin{split}
	m_-(t) & = 
	\sin \frac{at}{2} \cos \frac{bt}{2} + \kappa \sin\frac{bt}{2}\cos\frac{at}{2}\\
	n_-(t) & = 
b\sin \frac{at}{2} \cos \frac{bt}{2} -a \sin\frac{bt}{2}\cos\frac{at}{2}\\
	d_-(t) & = b\sin^2 \frac{at}{2} + \kappa a \sin^2 \frac{bt}{2}
\end{split}
\end{equation}
On the interval $0<t\leq 2\pi$ with $[t]_\pi\neq 0$ the derivatives of both $q_-(t)$ and $r_-(t)$ are zero when either $m_-(t)=0$ or $n_-(t)=0$.
We consider the two cases separately. In both cases we use the identity
\begin{equation}
	q_-(t) =
		\frac{
				b\tan\frac{at}{2}\left (1+\tan^2\frac{bt}{2}\right )
				-
				a\tan\frac{bt}{2}\left (1+\tan^2\frac{at}{2}\right )
		}
		{
			b \tan^2 \frac{at}{2}\left (1+\tan^2\frac{bt}{2}\right )
			+\kappa a  \tan^2\frac{bt}{2}\left (1+\tan^2\frac{at}{2}\right )
		}
\end{equation}

\begin{description}
\item[Case 1] Suppose $t_1$ satisfies $m_-(t_1)=0$. 
At these values
\begin{equation}q_-(t_1) = \cot \frac{at_1}{2}\end{equation}
and
\begin{equation}\kappa q_-(t_1) =- \cot \frac{bt_1}{2}\end{equation}
Hence if we define
\begin{align}\label{rstar}
r_-^*(t) = b\left [ \frac{\pi}{2}- \frac{at}{2}\right ]_{[-\pi/2,\pi/2]}+a\left [-\frac{\pi}{2}+ \frac{bt}{2}\right ]_{[-\pi/2,\pi/2]}
\end{align}
for $t \in [0,2\pi]$ we find $r_-(t_1) = r_-^*(t_1)$
for all $t_1$ satisfying $m_-(t_1)=0$
The function $r_-^*(\cdot)$ is piecewise constant, taking values $(-a-b+2\lambda)\pi/2$ with $\lambda = 1,\ldots,a+b-1$. On each piecewise constant interval there is a $t_1$ satisfying $m_-(t_1)=0$. Hence these turning points of $r_-(t)$ lie within the bounds $\pm(a+b-2)\frac{\pi}{2}$ with at least one on the bound.

\begin{figure}[htbp]
\begin{center}
\includegraphics[width = 0.9\linewidth]{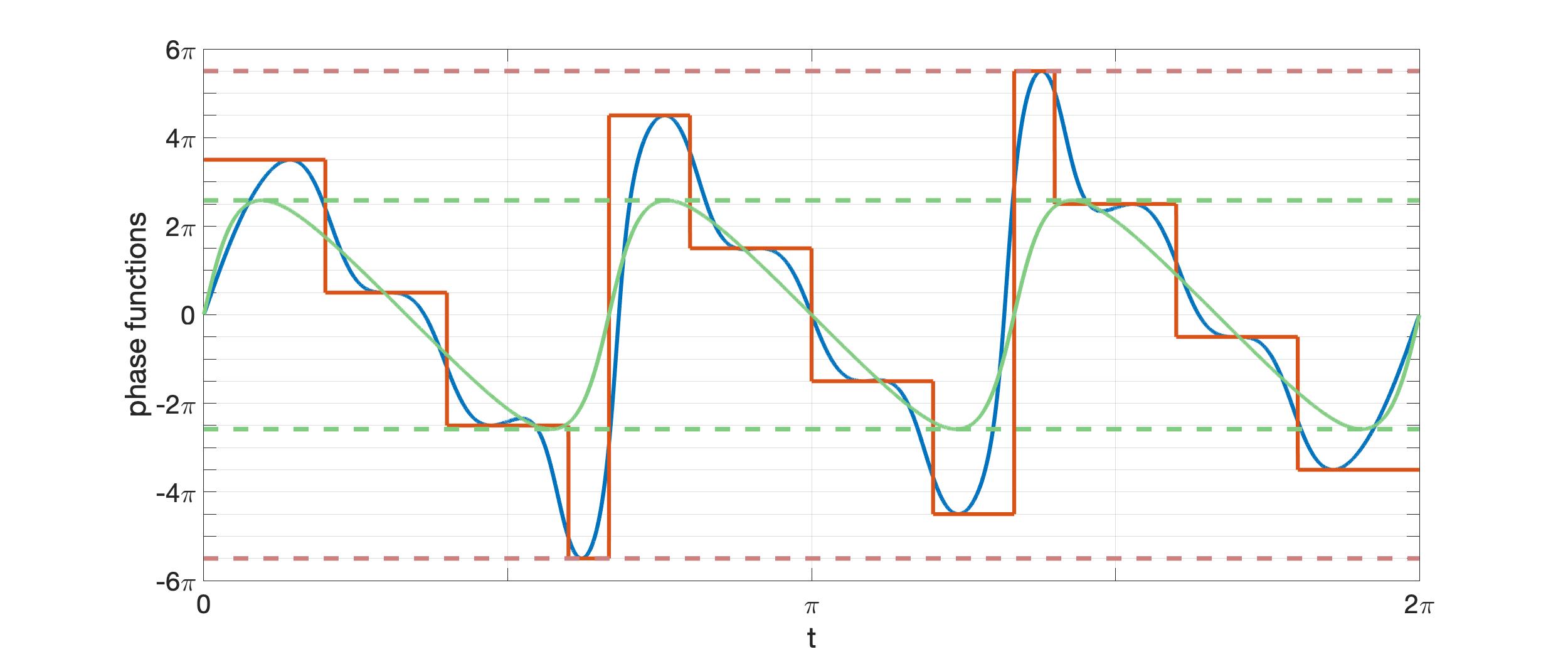}
\caption{Phase functions $r_-$ (blue), $r_-^*$ (red) and $r_-^\dagger$ (green) with $a=3$ and $b=10$. The turning points of $r_-$ where $m_-(t)=0$ take the value $(a+b-2\lambda)\pi/2$ with $\lambda$ an integer. The function $r_-^*(\cdot)$ is piecewise constant and takes these same values.  The turning points of $r_-$ where $n_-(t)=0$ take the values of $r_-^\dagger $, whose bounds are also shown.}\label{test23a}
\end{center}
\end{figure}
\item[Case 2] Define
\begin{equation}q^\dagger_-(t) = 
		\frac{
				(b^2-a^2)\sin at_2
		}
		{
			a^2+b^2+\kappa a b -(b^2-a^2)\cos at_2
		}
\end{equation}
and
\begin{equation}r^\dagger_-(t) = b \arctan q^\dagger_-(t)  + a \arctan \kappa q^\dagger_-(t).\end{equation}
Then $q_-(t_2)=q^\dagger_-(t_2)$ and  $r_-(t_2)= r_-^\dagger(t_2)$ for all $t_2$ satisfying  $n_-(t_2)=0$. It follows that $|r_-(t_2)|\leq |\bar{r}^\dagger|$ for all such $t_2$ where
\begin{equation}\label{def_rbar}
\begin{split}
\bar{r}^\dagger & =  b \arctan \bar{q}^\dagger + a \arctan \kappa \bar{q}^\dagger\\
\bar{q}^\dagger & =  
		   \frac{b^2-a^2}{2\sqrt{ab(a+\kappa b)(b+\kappa a)}}
\end{split}
\end{equation}

With some abuse of notation, write $\bar{r}^\dagger=\bar{r}^\dagger(\kappa)$; i.e. consider $\bar{r}^\dagger$ as a function of $\kappa$. We find
\begin{equation}\begin{split}
\frac{d}{d\kappa}\bar{r}^\dagger(\kappa) = &
\frac{-(a+b\kappa )(a^2-b^2)^2
	}
	{(2ab+ (a^2+b^2)\kappa)(2ab\kappa +a^2+b^2)
	} \\
& \times
\sqrt{
\frac{ab}{(a+b\kappa)(a\kappa+b)}
}\end{split}\end{equation}
Hence $|\bar{r}^\dagger(\kappa)| \leq \max(|\bar{r}^\dagger(0)|,\lim_{\kappa\rightarrow \infty}|\bar{r}^\dagger(\kappa)|)$. 
Furthermore
\begin{equation}\begin{split}
\bar{r}^\dagger(0) & =  b\arctan \left (\frac{b^2-a^2}{2ab}\right )\\
\lim_{\kappa\rightarrow \infty} \bar{r}^\dagger(\kappa) & = a\arctan \left (\frac{b^2-a^2}{2ab}\right )
\end{split}\end{equation}
Hence it suffices to show 
\begin{equation}\max(a,b) \arctan \left |\frac{b^2-a^2}{2ab}\right | \leq (a+b-2)\frac{\pi}{2}\end{equation}
If both $a$ and $b$ are both greater than 1 then this is immediate, since in this case $\max(a,b)\leq a+b-2$. Hence it suffices to show
\begin{equation} b \arctan \frac{b^2-1}{2b} \leq  (b-1)\frac{\pi}{2}\end{equation}
or equivalently, with $b\geq 2$, that
\begin{equation}  \frac{b^2-1}{2b}\sin\left ( \frac{\pi}{2b}\right ) \leq \cos\left ( \frac{\pi}{2b}\right )\end{equation}
We can quickly check
\begin{equation}\frac{b^2-1}{2b}\sin\left ( \frac{\pi}{2b}\right ) \leq \frac{(b^2-1)\pi}{4b^2}
\leq 
1 - \frac{\pi^2}{8b^2}
\leq
 \cos\left ( \frac{\pi}{2b}\right )
\end{equation}
\end{description}

\end{proof}


\begin{proof}[Proof of Theorem \ref{Meg_equiv_b}]
The proof is similar to that for Theorem~\ref{Meg_equiv_a}.
We have already established appropriate bounds for $r_-(t)$.
If we define 
\begin{equation}r_+(t) = b \arctan q_+(t) + a \arctan \kappa q_+(t)\end{equation}
then we need to show it is also bounded appropriately. 
Similar to the previous case, the turning points of $r_+(t)$ occur at the same values of $t$ as the turning points of $q_+(t)$. 
When $[t]_\pi\neq 0$ the derivative of $q_+(t)$ is given by
\begin{equation}\frac{d}{dt}q_+(t)  =ab \frac{m_+(t)n_+(t)}{d_+(t)^2}\end{equation}
with
\begin{equation}\begin{split}
	m_+(t) & = 
	\kappa \sin \frac{at}{2} \cos \frac{bt}{2} +  \sin\frac{bt}{2}\cos\frac{at}{2}\\
	n_+(t) & = 
b \sin\frac{bt}{2}\cos\frac{at}{2}
- a\sin \frac{at}{2} \cos \frac{bt}{2} \\
	d_+(t) & = b\cos^2 \frac{at}{2} + \kappa a \cos^2 \frac{bt}{2}
\end{split}\end{equation}
We will consider the cases $m_+(t)=0$ and $n_+(t)=0$ separately. This time we use the identity 
\begin{equation}\begin{split}
	q_+(t) & = 
		\frac{
				b\tan\frac{at}{2}\left (1+\tan^2\frac{bt}{2}\right )
				-
				a\tan\frac{bt}{2}\left (1+\tan^2\frac{at}{2}\right )
		}
		{
			b \left (1+\tan^2\frac{bt}{2}\right )
			+\kappa a  \left (1+\tan^2\frac{at}{2}\right )
		}
\end{split}\end{equation}
\begin{description}
\item[Case 1] Suppose $t_1$ satisfies $m_+(t_1)=0$. Then
\begin{equation}\begin{split}
	q_+(t_1) 
		& = \tan \frac{at_1}{2}
\end{split}\end{equation}
and
\begin{equation}\begin{split}
	\kappa q_+(t_1) 
		& = - \tan \frac{bt_1}{2}
\end{split}\end{equation}
Hence if we define
\begin{align}\label{rstar}
r_+^*(t) = b\left [  \frac{at}{2}\right ]_{[-\pi/2,\pi/2]}-a\left [ \frac{bt}{2}\right ]_{[-\pi/2,\pi/2]}
\end{align}
for $t \in [0,2\pi]$ we find $r_+(t_1) = r_+^*(t_1)$
for all $t_1$ satisfying $m_+(t_1)=0$.
The function $r_+^*(\cdot)$ is piecewise constant, taking values $(-a-b-1+2\lambda)\pi/2$ with $\lambda = 1,\ldots,a+b$ when either $a$ or $b$ are even, and values $(-a-b+2\lambda)\pi/2$ with $\lambda = 1,\ldots,a+b-1$ when $a$ and $b$ are both odd. On each piecewise constant interval there is a $t_1$ satisfying $m_+(t_1)=0$. Hence these turning points of $r_+(t)$ lie within the bounds $\pm(a+b-1)\frac{\pi}{2}$ (if either $a$ or $b$ even) or $\pm(a+b-2)\frac{\pi}{2}$ (if $a$ and $b$ both odd) with at least one on the bound.

\begin{figure}[htbp]
\begin{center}
\includegraphics[width = 0.9\linewidth]{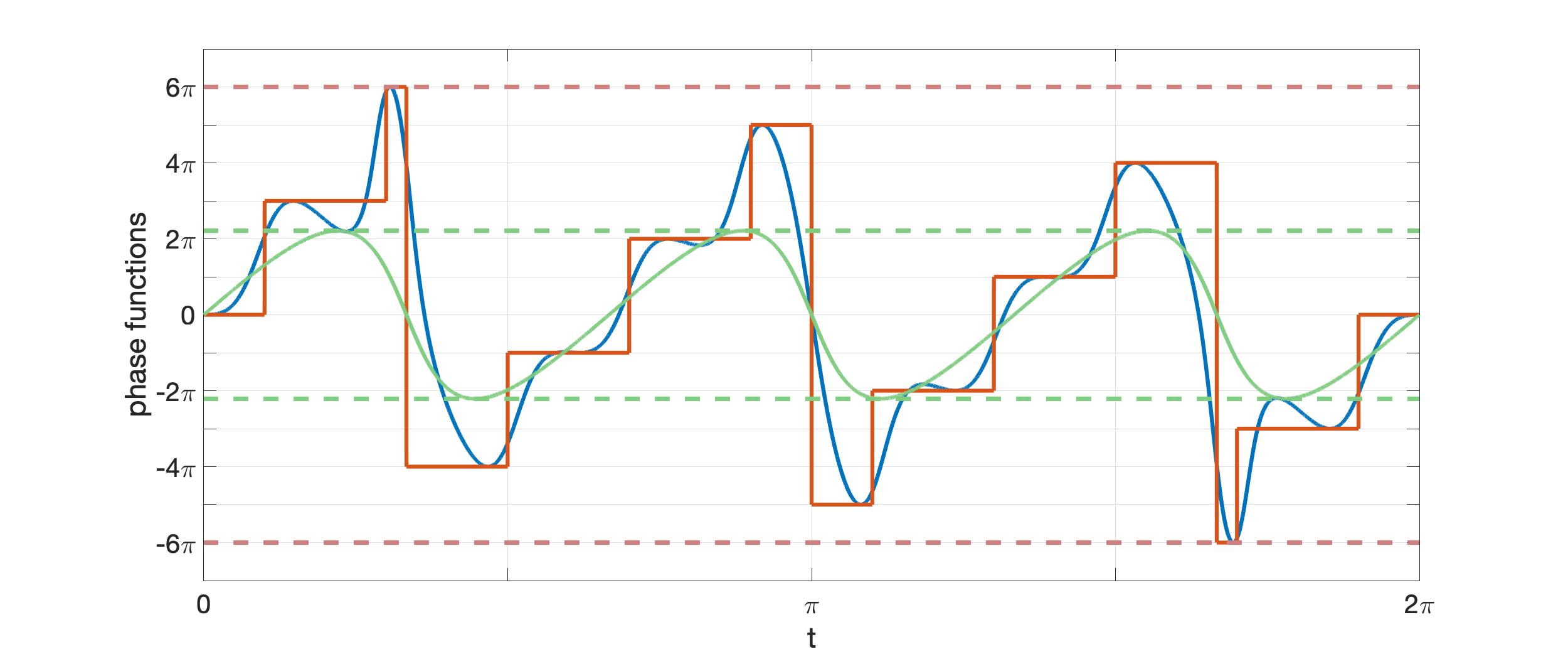}
\caption{
Phase functions $r_+$ (blue), $r_+^*$ (red) and $r_+^\dagger$ (green) with $a=3$ and $b=10$. The turning points of $r_+$ where $m_+(t)=0$ take the value $(a+b+1-2\lambda)\pi/2$ with $\lambda$ an integer. The function $r_+^*(\cdot)$ is piecewise constant and takes these same values.  The turning points of $r_+$ where $n_+(t)=0$ take the values of $r_+^\dagger $, whose bounds are also shown.
}\label{test25}
\end{center}
\end{figure}

\item[Case 2] 
Define
\begin{equation}q^\dagger_+(t) = 
		\frac{
				(b^2-a^2)\sin at_2
		}
		{
			a^2+b^2+\kappa a b +(b^2-a^2)\cos at_2
		}
\end{equation}
and
\begin{equation}r^\dagger_+(t) = b \arctan q^\dagger_+(t)  + a \arctan \kappa q^\dagger_+(t).\end{equation}
Then $q_+(t_2)=q^\dagger_+(t_2)$ and  $r_+(t_2)= r_+^\dagger(t_2)$ for all $t_2$ satisfying  $n_+(t_2)=0$. It follows that $|r_+(t_2)|\leq |\bar{r}^\dagger|$ for all such $t_2$ where $\bar{r}^\dagger$ is given by (\ref{def_rbar}).
%
%
As we have the same bounds as before, the previous analysis establishes that these turning points lie within the bounds.
\end{description}

\end{proof}